\documentclass{article} 
\usepackage[utf8]{inputenc} 

\usepackage{amsmath,amssymb,amsfonts}
\usepackage{amsthm,pstricks,enumitem, textcomp, xspace}
\usepackage{todonotes}
\usepackage{enumitem}

\usepackage{fancyhdr} 
\usepackage{hyperref}
\newcommand\myshade{85}
\colorlet{mylinkcolor}{red}
\colorlet{mycitecolor}{blue}
\colorlet{myurlcolor}{orange}

\hypersetup{
  linkcolor  = mylinkcolor!\myshade!black,
  citecolor  = mycitecolor!\myshade!black,
  urlcolor   = myurlcolor!\myshade!black,
  colorlinks = true,
}

\newtheorem{theorem}{Theorem}[section]

\newtheorem{lemma}[theorem]{Lemma}
\newtheorem{definition}[theorem]{Definition}
\let\olddefinition\definition
\renewcommand{\definition}{\olddefinition\normalfont}

\newtheorem{example}[theorem]{Example}
\let\oldexample\example
\renewcommand{\example}{\oldexample\normalfont}
\newtheorem{proposition}[theorem]{Proposition}
\newtheorem{claim}[theorem]{Claim}
\newtheorem{fact}{Fact}
\let\oldfact\fact
\renewcommand{\fact}{\oldfact\normalfont}

\newtheorem{observation}[theorem]{Lemma}
\newtheorem{remark}[theorem]{Remark}
\let\oldremark\remark
\renewcommand{\remark}{\oldremark\normalfont}

\newcommand{\Obs}{Lemma\xspace}

\newcommand{\CCC}{CAT(0) cube complex\xspace}
\newcommand{\inseparable}[1]{#1-inseparable\xspace}
\newcommand{\elemfoldable}{elementary foldable\xspace}

\newcommand{\Z}{\mathbb Z }

\newcommand{\condX}{Property CC\xspace}

\newcommand{\modsim}{/\hspace{-.15cm}\sim}
\newcommand{\actson}{\curvearrowright} 
\newcommand{\immersedin}{\looparrowright} 
\newcommand{\embeddedin}{\hookrightarrow} 
\newcommand{\Stab}{\mathrm{Stab}} 

\newcommand{\gp}[1]{#1} 

\newcommand{\CC}[1]{\mathbf{#1}} 
\newcommand{\CCv}[1]{\mathbf{#1}} 
\newcommand{\CCc}[1]{\mathbf{#1}} 

\newcommand{\Hyp}[1]{\hat{\mathcal{#1}}} 
\newcommand{\hyp}[1]{\hat{\mathfrak{#1}}} 

\newcommand{\Hs}[1]{\mathcal{#1}} 
\newcommand{\hs}[1]{\mathfrak{#1}} 
\newcommand{\comp}[1]{{#1}^*} 
\newcommand{\embedsin}{{{}^|\hspace{-0.19cm}\longrightarrow}}

\newcommand{\simp}[1]{#1} 

\newcommand{\usimp}[1]{\tilde{#1}} 

\newcommand{\trk}[1]{#1} 

\author{Benjamin Beeker\thanks{The first author was supported by ISF grant 1941/14.} ~and Nir Lazarovich\thanks{The second author is supported by ETH Zürich Postdoctoral Fellowship Program.}}

\title{Stallings' folds for cube complexes}
\date{} 
         
\begin{document}
\maketitle
\begin{abstract}
We describe a higher dimensional analogue of Stallings' folding sequences for group actions on CAT(0) cube complexes. We use it to give a characterization of quasiconvex subgroups of hyperbolic groups that act properly co-compactly on CAT(0) cube complexes via finiteness properties of their hyperplane stabilizers.
%
\end{abstract}


\section{Introduction}

By the construction developed by Sageev \cite{Sag95} and results of Bergeron-Wise \cite{BeWi12}, given a hyperbolic group $G$ with ``enough'' quasiconvex codimension-1 subgroups, one can construct a CAT(0) cube complex on which $G$ acts properly and cocompactly. Analogous methods were used extensively to find cubulations of many hyperbolic groups including: 
Coxeter groups \cite{NiRe03}, 
small cancellation groups \cite{Wis04}, 
hyperbolic 3-manifolds \cite{KaMa12}, 
certain free-by-cyclic hyperbolic groups \cite{HaWi15},
random groups at low densities \cite{OlWi11},
malnormal amalgams of cubulated hyperbolic groups \cite{HsWi15},
certain one relator groups with torsion \cite{LaWi13},
graphs of free groups with cyclic edge groups \cite{HsWi10},
etc.

An action on a CAT(0) cube complex often allows one to deduce algebraic and geometric consequences on the group in question. 
One example is the recent solution of the Haken conjecture using results of Agol \cite{Ago13}, Haglund-Wise \cite{HaWi08}, Wise \cite{Wis11} and others.
This was made possible through the study of quasiconvex subgroups of hyperbolic cubulated groups.
An important observation about quasiconvex subgroups of hyperbolic cubulated groups was made by Haglund in \cite{Hag08} who proved that a quasiconvex subgroup of a cubulated hyperbolic group has a convex subcomplex on which it acts cocompactly. 
It is easy to see that this result could serve as a characterization of quasiconvex subgroups in cubulated hyperbolic groups.
A similar statement for relatively hyperbolic groups was proved independently by Sageev and Wise \cite{SaWi15}.

A finitely generated group $H$ acting on a metric space $X$ is \emph{undistorted} if some (hence every) orbit map $H\to X$ is a quasi-isometry (where $H$ is endowed with the metric of shortest word with respect to a finite generating set), and is \emph{distorted} otherwise.
In the setting of a finitely generated subgroup $H$ in a finitely generated group $G$ one can consider distortion with respect to the length induced from finite generating sets of $H$ and $G$.
If $G$ is hyperbolic, a subgroup is undistorted if and only if it is quasi-convex.

A lot of research has been done on distortion of groups.
Perhaps the easiest way of finding a distorted subgroup of a hyperbolic group is to find an infinite normal finitely generated subgroup of infinite index.
For example the surface fiber subgroup of a fibered hyperbolic 3-manifold group.
Rips \cite{Rip82} showed how to construct finitely generated normal subgroup of hyperbolic groups using a small cancellation construction.
A similar construction was carried by Wise \cite{Wis98} arranging such that the ambient group is cubulated (in fact, it is the fundamental group of a compact 2-dimensional non-positively curved square complex).
Dison and Riley \cite{DiRi13} construct examples of groups, called Hydra groups, that are fundamental groups of non-positively curved square complexes, and have very distorted free subgroups.
The distortion they achieve exceeds those found in previous works of Mitra \cite{Mit98} and the subsequent
2-dimensional CAT(-1) groups of Barnard, Brady and Dani \cite{BBD12}.

This paper aims to provide a characterization of quasiconvex subgroups via finiteness properties of their hyperplane stabilizers. 
%
Let $\Hyp{H}$ be the set of hyperplanes and $$\Hyp{IH} = \left\{ \bigcap _{i=1} ^n \hyp{h}_i \ne \emptyset\; \middle|\; n\ge 0, (\hyp{h}_1,\dots \hyp{h}_n) \in\Hyp{H}^n \right\}$$ be the set of Intersections of collections of pairwise transverse Hyperplanes.
We prove the following characterization of quasiconvex subgroups of hyperbolic cubulated groups.
\begin{theorem}
\label{thm:QC}
Let $G$ be a hyperbolic group acting properly and co-compactly on a finite dimensional CAT(0) cube complex $\CC{X}$. Let $H\le G$ be a finitely presented subgroup. Then the following are equivalent:
\begin{enumerate}
\item \label{1} The subgroup $H$ is quasiconvex in $G$.
\item \label{2} For all $\hyp{t}\in \Hyp{IH}$, the group $\Stab_H (\hyp{t})$ is finitely presented.
\item \label{3} For all $\hyp{k}\in\Hyp{H}$, $\Stab_H (\hyp{k})$ is quasiconvex in $G$.
\item \label{4} The subgroup $H$ is hyperbolic and for all $\hyp{k}\in\Hyp{H}$, $\Stab_H (\hyp{k})$ is quasiconvex in $H$.
\end{enumerate}
\end{theorem}

Before discussing the proof of the theorem, let us examine it in an example. Let $G$ be the fundamental group of a closed hyperbolic $3$-manifold which is fibered over the circle, and let $H$ be the fundamental group of the surface fiber.
As we mentioned before $H$ is distorted, and forms a short exact sequence $1\rightarrow H\rightarrow G \rightarrow \Z\rightarrow 1$. 
Consider the cubulation of $G$ obtained in \cite{KaMa12,BeWi12}.
By construction, the stabilizer $L=\Stab_G(\hyp{t})$ of the hyperplane $\hyp{t}$ is a quasiconvex surface subgroup of $G$. 
Since both $L$ and $H$ are surface subgroups, if $L\le H$ then $[H:L]<\infty$, contradicting the fact that $L$ is undistorted. 
Hence $L/(H\cap L) \simeq \Z$ and thus $\Stab_H(\hyp{t}) = H\cap L$ is not finitely generated (since it corresponds to an infinite cyclic cover of a surface). This shows that (\ref{2}),(\ref{3}) and (\ref{4}) of Theorem \ref{thm:QC} do not hold (cf. Theorem \ref{thm: undistorted in 2D} which also applies to this case).
In fact, in this case we showed that \emph{every} hyperplane stabilizer in $H$ is not finitely generated.

The main tool in the proof of Theorem \ref{thm:QC} is a higher dimensional analogue of Stallings' folds.

In his seminal paper \cite{Sta83} Stallings introduced the notion of Stallings' folds in order to study finitely generated subgroups of free groups.
Given a finitely generated subgroup of a free group and a generating set, Stallings' folds provide an algorithm to replace the generating set with a minimal generating set.
His main observation is that given any combinatorial map between finite graphs, one can find a finite sequence of identifications of adjacent edges, called folds, of the domain graph such that the induced map from the resulting folded graph would be a local isometry.
In \cite{Sta91}, he extended this idea to more general G-trees.

One useful property of these foldings is the following (see \cite[Proposition p.455]{BeFe91}).
Let $G$ be a finitely generated group. Let $f$ be a $G$-equivariant simplicial map of G-trees $T\to T'$ sending edges to edges, and assume the action on $T'$ has finitely generated edge stabilizers. 
Then one can perform finitely many $G$-equivariant folds of $T$ such that the induced map on the resulting tree $T''\to T'$ is a $G$-equivariant isometric embedding, and the map $f$ is the composition of the folds and of the embedding. In \cite{BeFe91}, Bestvina and Feighn applied this property in the case where $T$ is the Dunwoody resolution of the tree $T'$ (see \cite{Dun85}).

In \cite{BeLa15}, we described a generalization of resolutions in the setting of CAT(0) cube complexes. The construction can be summarized as follow.
Let $G$ be a finitely presented group, and let $K$ be its presentation complex. Let $\CC{X}'$ be a cube complex on which $G$ acts. We build a $G$-equivariant map from $\tilde{K}$, the universal cover of $K$, to the CAT(0) cube complex $\CC{X}'$. A connected component of the preimage of a hyperplane is called a track, and can be seen as an embedded graph in $\tilde{K}$. It defines a wall on $K$, and the set of all such walls defines a CAT(0) cube complex $\CC{X}$ endowed with a natural action of $G$ and a $G$-equivariant map to $\CC{X}'$. The construction and the properties of resolutions are described more
thoroughly in \cite{BeLa15}. 


In this paper, we introduce an analogue of Stallings' folds of $G$-trees in the context of CAT(0) cube complexes.

Given a finitely presented subgroup $H$ of a cubulated hyperbolic group $G\actson \CC{X}'$, we first resolve the cube complex and get an $H$-equivariant map $\CC{X}\to\CC{X}'$, where $\CC{X}$ is the geometric resolution of the action of $H$ on $\CC{X}'$. 
We then provide conditions for factoring this resolution through a finite sequence of folds until the resulting folded complex embeds into the cubulation of $G$ with respect to the combinatorial metric on cube complexes. 
We denote such a folding sequence by
$$ \CC{X}=\CC{X}_0 \leadsto \CC{X}_1 \leadsto \ldots \leadsto \CC{X}_n\embedsin\CC{X}'$$
where each $\leadsto$ is an elementary fold, and the map $\embedsin$ is an $L^1$ embedding of cube complexes.
Finally we show that under some assumptions $H$ acts coboundedly on the resulting folded complex $\CC{X}_n$, from which we deduce that $H$ is quasiconvex in $G$.

In general, it is not clear whether one can replace ``finitely presented'' with ``finitely generated'' in Theorem \ref{thm:QC}.
Since our methods use the geometric resolution, which is only defined for finitely presented groups, we were unable to treat these cases. However,
if the group is a surface group or if the cube complex is two dimensional then we obtain the following easier criterion for undistortion.

\begin{theorem}\label{thm: undistorted in 2D}
Let $G$ be a finitely presented group that acts properly on a CAT(0) cube complex $\CC{X}'$, with finitely generated hyperplane stabilizers, and if one of the following holds:
\begin{enumerate}
\item \label{surface_group_case}the group $G$ is a surface group or a free group, or
\item \label{2_dim_case}the complex $\CC{X}'$ is 2 dimensional,
\end{enumerate}
then the orbit maps $G\to \CC{X}'$ are quasi-isometric embeddings.
\end{theorem}

We remark that in the case of trees it was sufficient for the map to be a local isometric embedding for it to be a (global) isometric embedding. While such a statement is true for the CAT(0) metric on the cube complex (i.e, the $L^2$ metric), it is not true for the $L^1$ metric.
On the other hand, in the setting of the above theorem we are forced to use the $L^1$ metric, since having an $L^2$ combinatorial embedding would imply the existence of a convex cocompact subcomplex core for $G$. Such a subcomplex does not exist even for the simple example of the cyclic group generated by the translation by $(1,1)$ on $\mathbb{R}^2$ with its standard tiling by unit squares.


Theorem \ref{thm: undistorted in 2D} Case \ref{2_dim_case} has been independently proved using similar ideas of Stallings' folds for VH square complexes in Samuel Brown's PhD thesis \cite{Bro16}.

In some sense, this paper can be considered as a continuation of our previous paper  \cite{BeLa15} on resolutions and finiteness properties of resolutions, in that we try to study the properties of the resolution map via Stallings' folding sequences.
%

\subsection*{Acknowledgments}
The authors would like to thank Michah Sageev for his helpful comments on this manuscript, and to Dani Wise for fruitful discussions. We would also like to thank the referee for spotting a mistake in a previous version of the paper and for helpful suggestions. The second author acknowledges the support received by the ETH Zurich Postdoctoral Fellowship Program and the Marie Curie Actions for People COFUND Program.


\section{Preliminaries }\label{preliminaries}

\subsection{CAT(0) cube complexes}
We begin by a short survey of definitions concerning CAT(0) cube complexes. For further details see, for example, \cite{Saa12}.

A \emph{cube complex} is a collection of euclidean unit cubes of various dimensions in which faces have been identified isometrically. 

A simplicial complex is \emph{flag} if every $(n+1)$-clique in its 1-skeleton spans a $n$-simplex.
A cube complex is \emph{non-positively curved} (NPC) if the link of every vertex is a flag simplicial complex. It is a \emph{\CCC} if moreover it is simply connected.

A cube complex $\CC{X}$ can be equipped with two natural metrics, the euclidean and the $L^1$ (or combinatorial) metric. With respect to the former, the complex $\CC{X}$ is NPC if and only if it is NPC \`{a} la Gromov (see \cite{Gro87}). However, the latter is more natural to the combinatorial structure of CAT(0) cube complexes described below.

Given a cube $\CCc{C}$ and an edge $\CCc{e}$ of $\CCc{C}$. The midcube of $\CCc{C}$ associated to $\CCc{e}$ is the convex hull of the midpoint of  $\CCc{e}$ and the midpoints of the edges parallel to $\CCc{e}$.
The \emph{hyperplane} associated to $\CCc{e}$ is the smallest subset containing the midpoint of $\CCc{e}$ and such that if it contains a midpoint of an edge it also contains all the midcubes containing it.
Every hyperplane $\hyp{h}$ in a \CCC $\CC{X}$ separates $\CC{X}$ into exactly two components (see \cite{NiRe98}) called the \emph{halfspaces} associated to $\hyp{h}$. 
Thus, a hyperplane can also be abstractly identified with its pair of complementary halfspaces. 
For a \CCC $\CC{X}$ we denote by $\Hyp{H}=\Hyp{H}(\CC{X})$ the set of all hyperplanes in $\CC{X}$, and by $\Hs{H}=\Hs{H}(\CC{X})$ the set of all halfspaces. 
For each halfspace $\hs{h}\in \Hs{H}$ we denote by $\comp{\hs{h}}\in\Hs{H}$ its complementary halfspace, and by $\hyp{h}\in\Hyp{H}$ its bounding hyperplane, which we also identify with the pair $\{\hs{h},\comp{\hs{h}}\}$. Conversely, a choice of a halfspace $\hs{h}$ for a hyperplane $\hyp{h}$ is called an \emph{orientation} of $\hyp{h}$.
We denote the inclusion of halfspaces by $\le$.

We briefly review the terminology that will be used throughout the paper.


Two distinct hyperplanes $\hyp{h},\hyp{k}\in\Hyp{H}$ can be either \emph{disjoint} or \emph{transverse}. The latter is denoted by $\hyp h \pitchfork \hyp k$. 

Two distinct halfspaces $\hs{h},\hs{k}\in\Hs{H}$ can be in one of the following arrangements:
\begin{description}
	\item [nested] if $\hs{h}\le\hs{k}$ or $\hs{k}\le\hs{h}$.
	\item [facing] if $\hs{h}>\comp{\hs{k}}$, or equivalently, if both $\hyp{h}\subset\hs{k}$ and $\hyp{k}\subset \hs{h}$.
	\item [transverse] if $\hyp{h}$ and $\hyp{k}$ are transverse. In this case, we denote $\hs{h} \pitchfork \hs{k}$.
	\item [incompatible] if $\hs{h}$ and $\hs{k}$ have empty intersection, or equivalently, if $\comp{\hs{h}}\ge\hs{k}$. Otherwise, they are said to be \emph{compatible}.
\end{description}

A hyperplane in a \CCC \emph{separates} two points if they belong to different halfspaces of the hyperplane. 
A hyperplane $\hyp{h}$ \emph{separates} two hyperplanes $\hyp{h'}$ and $\hyp{h''}$ if it separates any point of $\hyp{h'}$ from any point of $\hyp{h''}$, or equivalently if there exists an orientation of each of them such that $\hs{h'} < \hs{h} < \hs{h''}$.
\begin{definition}[\inseparable{$\Hyp{A}$} and facing]
Given a set of hyperplanes $\Hyp{A}$, two distinct hyperplanes are \emph{\inseparable{$\Hyp{A}$}}, if no hyperplane in $\Hyp{A}$ separates them.
The collection $\Hyp{A}$ is said to be \emph{facing} if any two distinct hyperplanes in $\Hyp{A}$ are disjoint and \inseparable{$\Hyp{A}$}, or equivalently if the hyperplanes have an orientation for which every pair is facing.
\end{definition}

\begin{remark} \label{existence of inseparable}
	Since any two hyperplanes in a CAT(0) cube complex are separated by finitely many hyperplanes, for every non empty set of hyperplanes $\Hyp{A}$,  and every hyperplane $\hyp{h}$, there exists a hyperplane $\hyp{k}\in\Hyp{A}$ such that $\hyp{h}$ and $\hyp{k}$ are \emph{\inseparable{$\Hyp{A}$}}.
\end{remark}

\subsection{Pocsets to \CCC}\label{pocsetstoccc}

We adopt Roller's viewpoint of Sageev's construction. Recall from \cite{Rol98} that a \emph{pocset} is a triple $(\Hs{P},\le,\comp{})$ of a poset $(\Hs{P},\le)$ and an order reversing involution $\comp{}:\Hs{P}\to\Hs{P}$ satisfying $\hs{h}\neq\comp{\hs{h}}$ and $\hs{h}$ and $\comp{\hs{h}}$ are incomparable for all $\hs{h}\in\Hs{P}$. A pocset is \emph{locally finite} if for any pair of elements, the set of elements in between them is finite. In what follows we assume that all pocsets are locally finite.

The set of halfspaces $\Hs{H}$ of a \CCC has a natural pocset structure given by the inclusion relation and the complement operation $\comp{}$. Roller's construction starts with a locally finite pocset $(\Hs{P},\le,\comp{})$ of finite width  (see \cite{Saa12} for definitions) and constructs a \CCC $\CC{X}(\Hs{P})$ such that $(\Hs{H}(\CC{X}),\le, \comp{})=(\Hs{P},\le,\comp{})$. We briefly recall the construction, for more details see \cite{Rol98} or \cite{Saa12}.

An \emph{ultrafilter} $U$ on $\Hs{P}$ is a subset verifying $\# \left(U\cap \left\{\hs{k},\comp{\hs{k}}\right\}\right) = 1$ for all $\hs{k}\in \Hs{P}$ and such that for all $\hs{h} \in U$, if $\hs{h} \leq \hs{k}$ then $\hs{k}\in U$. If we denote $\hyp{h} =\{ \hs{h},\comp{\hs{h}} \}$ and  $\Hyp{P}=\left\{\hyp{h} \middle| \hs{h}\in\Hs{P} \right\}$, then $U$ can be viewed as a choice function $U:\Hyp{P}\to\Hs{P}$. Throughout the paper we will use both viewpoints.

An ultrafilter $U$ satisfies the \emph{Descending Chain Condition} (DCC) if any descending chain $\hs{k}_1 > \hs{k}_2 > \dots> \hs{k}_n > \dots$  of element of $U$ has finite length. 
The vertices of $\CC{X}(\Hs{P})$ are the DCC ultrafilters of $\Hs{P}$. 
Two vertices of $\CC{X}(\Hs{P})$ are connected by an edge if the corresponding ultrafilters differ on a single pair in $\Hyp{P}=\left\{ \{ \hs{h},\comp{\hs{h}} \} \middle| \hs{h}\in\Hs{P} \right\}$. 
An $n$-cube is added to every one skeleton of an $n$-cube. Or equivalently, any $n$-cube corresponds to $2^n$ distinct DCC ultrafilters that differ on a set of $n$ hyperplanes in $\Hyp{P}$.

We remark that an ultrafilter can be defined equivalently as a subset of $\Hs{H}$ whose elements are pairwise compatible and it is maximal for this property.


\section{Quotients of pocsets}

The basic construction that enables one to fold hyperplanes is the introduction of quotients of CAT(0) cube complexes.
The details of this construction will be given in the language of pocsets, and thus might seem cumbersome and lacking a geometric intuition. However the basic idea remains quite simple; given an admissible equivalence relation $\sim$ on the hyperplanes of a CAT(0) cube complex $\CC{X}$ we would like to introduce a pocset structure on the quotient $\Hs{H}(\CC{X})/\hspace{-.15cm}\sim$ such that the dual CAT(0) cube complex $\CC{X}/\hspace{-.15cm}\sim:=\CC{X}(\Hs{H}(\CC{X})/\hspace{-.15cm}\sim)$ is the `smallest' CAT(0) cube complex to have a combinatorial map  $\CC{X}\to\CC{X}/\hspace{-.15cm}\sim$ for which the preimages of hyperplanes are equivalence classes of $\sim$.
To explain the geometric outcome of this construction we illustrate it with an example.
The example also shows an interesting phenomenon that may occur; the dimension of the cube complex can increase when passing to a quotient.

\begin{example}\label{example of quotient}
Consider the CAT(0) cube complex $\CC{X}$ shown on the top of Figure \ref{square}. 
The hyperplanes of $\CC{X}$ are shown as colored midcubes. 
Let $\sim$ be the equivalence relation on the hyperplanes of $\CC{X}$ whose classes are shown by colors, that is, two hyperplanes are $\sim$-equivalent if they have the same color.
The quotient CAT(0) cube complex $\CC{X}/\hspace{-.15cm}\sim$ is shown on the bottom of Figure \ref{square} and the image of the map $\CC{X}\to\CC{X}/\hspace{-.15cm}\sim$  is shown with bold lines.


	\begin{figure}[ht!]
		\begin{center}
			\def\svgwidth{0.7\textwidth}
			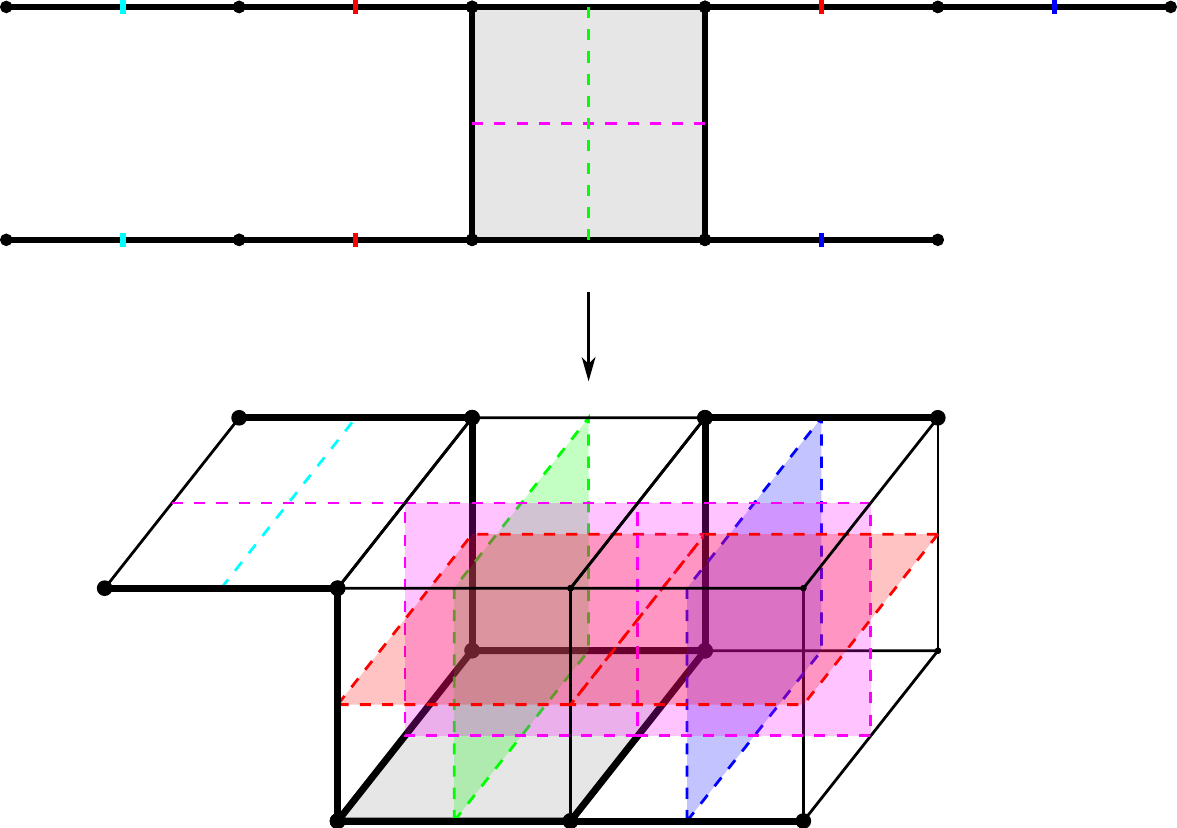
			\caption{The CAT(0) cube complex $\CC{X}$ and the quotient $\CC{X}/\hspace{-.15cm}\sim$ obtained by identifying hyperplanes with the same color.}
			\label{square}
		\end{center}
	\end{figure}
\end{example}

For the quotient to carry a pocset structure one needs to restrict to a subclass of equivalence relations.

\begin{definition}
Let $(\Hs{H},\leq ,\comp{})$ be a pocset. An equivalence relation $\sim$ on $\Hs H$ is an \emph{admissible equivalence} if it satisfies the following $\forall \hs{h},\hs{k}\in \Hs{H}$:
\begin{enumerate}[label=(AER\arabic*)]
	\item \label{selfcomp}$\hs{h}\nsim\comp{\hs{h}}$,
	\item \label{compinv}if $\hs{h}\sim\hs{k}$ then $\comp{\hs{h}}\sim\comp{\hs{k}}$, \footnote{By properties \ref{selfcomp} and \ref{compinv} the equivalence relation $\sim$ on $\Hs{H}$ defines an equivalence relation on $\Hyp{H}$ which we also denote by $\sim$.}
	\item \label{selfcross}if $\hs{h}\pitchfork\hs{k}$ then $\hs{h}\nsim\hs{k}$, and 
	\item \label{welloriented}if $\hyp{h}\sim\hyp{k}$ are distinct and \inseparable{$[\hyp{h}]$} hyperplanes (where $[\hyp{h}]=[\hyp{k}]$ denotes the equivalence class\footnote{We will sometimes use a subscript notation $[\hyp{h}]=[\hyp{h}]_{\sim}$ if the equivalence relation is not clear from the context.} of $\hyp{h}$ and $\hyp{k}$), then $\hs{h}\sim\hs{k}$ where $\hs{k}$ and $\hs{h}$ are facing.
\end{enumerate}
\end{definition}

\begin{remark}
	An admissible equivalence relation $\Hs{H}$ is uniquely determined by its induced relation on $\Hyp{H}$. In fact, every equivalence relation on the hyperplanes $\Hyp{H}$ in which transverse hyperplanes are not equivalent determines a unique admissible equivalence relation on halfspaces $\Hs{H}$. 
\end{remark}

We define a pocset structure on $\Hs{H}/\hspace{-.15cm}\sim$ in the following way. 
\begin{itemize}
	\item The complementation $\comp{}:\Hs{H}/\hspace{-.15cm}\sim\to\Hs{H}/\hspace{-.15cm}\sim$ is defined by $\comp{[\hs{h}]}=[\comp{\hs{h}}]$.
	\item The poset structure is defined by $[\hs{h}]< [\hs{k}]$ if any choice of representatives $\hs{h},\hs{k}$ of $[\hs{h}],[\hs{k}]$ respectively have disjoint bounding hyperplanes $\hyp{h},\hyp{k}$, and, if moreover $\hyp{h},\hyp{k}$ are \inseparable{($[\hyp{h}] \cup [\hyp{k}]$)}, then $\hs{h}< \hs{k}$. We denote $[\hs{h}]\le[\hs{k}]$ if they are equal or $[\hs{h}]<[\hs{k}]$.
\end{itemize}

In Lemma \ref{lem: quotient by AER gives pocset} we will show that this definition indeed gives a pocset structure. For this  we will need the following two easy lemmas.

\begin{observation}
\label{observation}
Let $\hs{h},\hs{k}\in\Hs{H}$ be such that $[\hs{h}]<[\hs{k}]$ then
\begin{itemize}
\item If $\hyp{h},\hyp{k}$ are \inseparable{$[\hyp{k}]$} then $\hyp{h}\subset \hs{k}$.
\item Similarly, if $\hyp{h},\hyp{k}$ are \inseparable{$[\hyp{h}]$} then $\hyp{k}\subset \comp{\hs{h}}$.\qed
\end{itemize}
\end{observation}

\begin{observation}\label{unfolding2}
Given two halfspaces $\hs{h}$ and $\hs{l}$ and an equivalence class $[\hs{k}]$  such that $[\hs{h}] < [\hs{k}]$ and $[\hs{k}] < [\hs{l}]$. Then there is an element $\hyp{k}' \in [\hyp{k}]$ that separates $\hyp{h}$ and $\hyp{l}$. 
If moreover, $\hs{h}<\hs{l}$ then there exists $\hs{k}'\in[\hs{k}]$ such that $\hs{h}<\hs{k}'<\hs{l}$.
\end{observation}
\begin{proof}
Let $\hyp{k}$ be a representative of $[\hyp{k}]$ such that $\hyp{k}$ and $\hyp{h}$ are \inseparable{$[\hyp{k}]$}. 
By Remark \ref{existence of inseparable}, such an element exists and \Obs \ref{observation} implies that $\hs{k}$ contains $\hyp{h}$. There are two cases:
either $\hyp{k}$ and $\hyp{l}$ are \inseparable{$[\hyp{k}]$} 
which by \Obs \ref{observation} implies that $\comp{\hs{k}}$ contains $\hyp{l}$, which proves that $\hyp{k}$ separates $\hyp{h}$ and $\hyp{l}$, 
or there is a hyperplane $\hyp{k}'\in [\hyp{k}]$ such that $\hyp{k}'$ separates $\hyp{k}$ and $\hyp{l}$, but since $\hyp{k}'$ cannot intersect $\hyp{h}$ and cannot separate $\hyp{h}$ and $\hyp{k}$, it must also separate $\hyp{h}$ and $\hyp{l}$.
In both cases we found a hyperplane in $[\hyp{k}]$ that separates $\hyp{h}$ and $\hyp{l}$.

If moreover $\hs{h}<\hs{l}$, let $\hyp{k}'$ be a hyperplane that separates $\hyp{h}$ and $\hyp{l}$ and such that $\hyp{k}'$ and $\hyp{h}$ are \inseparable{$[\hyp{k}]$}. Then from Lemma \ref{observation} and $\hs{h}<\hs{l}$ it follows that $\hs{h}<\hs{k}'<\hs{l}$.
\end{proof}
%

\begin{lemma}\label{lem: quotient by AER gives pocset}
Let $(\Hs{H},\leq ,\comp{})$ be a locally finite pocset, and let $\sim$ be an admissible equivalence relation on $\Hs{H}$. Then, the triple $(\Hs{H}/\hspace{-.15cm}\sim,\leq ,\comp{})$ is a locally finite pocset.
\end{lemma}
\begin{proof}

Let us begin by proving that $(\Hs{H}/\hspace{-.15cm}\sim,\leq )$ is a poset. The reflexivity is clear from the definition. 

The relation $\le$ is antisymmetric since if $[\hs{h}]< [\hs{k}]$ and $[\hs{k}]<[\hs{h}]$ then any representatives $\hs{h},\hs{k}$ of $[\hs{h}],[\hs{k}]$ respectively such that $\hyp{h},\hyp{k}$ are \inseparable{($[\hyp{h}]\cup [\hyp{k}]$)} we must have $\hs{h}<\hs{k}< \hs{h}$ which contradicts the antisymmetry of $(\Hs{H},<)$.


The relation $\le$ is transitive: let $[\hs{h}]< [\hs{k}]<[\hs{l}]$ and let $\hyp{h}, \hyp{l}$ be representatives of $[\hyp{h}],[\hyp{l}]$ respectively. 
By \Obs \ref{unfolding2}, the hyperplanes $\hyp{h}$ and $\hyp{l}$ are separated by a hyperplane in $[\hyp{k}]$ and thus in particular disjoint. 

Now, assume furthermore, that $\hyp{h}$ and $\hyp{l}$ are \inseparable{($[\hyp{h}] \cup [\hyp{l}]$)}.
Let $\hyp{k}$ be a representative of $[\hyp{k}]$ that separates $\hyp{h}$ and $\hyp{l}$ such that $\hyp{k}$ and $\hyp{h}$ are \inseparable{$[\hyp{k}]$}. 
 By Remark \ref{existence of inseparable} such an element exists. By \Obs \ref{observation} and since $\hyp{k}$ and $\hyp{l}$ are on the same side of $\hyp{h}$, we deduce that $\hyp{l}\subset \comp{\hs{h}}$. Similarly, $\hyp{h}\subset \hs{l}$. Thus $\hs{h}< \hs{l}$.


The complementary operation is well defined by property \ref{compinv}, and defines an involution such that $[\hs{h}]\ne\comp{[\hs{h}]}$ because of property \ref{selfcomp}.
By the definition of the poset it follows that  $[\hs{h}]$ is incomparable with $\comp{[\hs{h}]}$ and that the complementary operation is order reversing. It is thus a pocset.

We now prove that the pocset is locally finite. Let $[\hs{h}]\leq [\hs{l}]\in\Hs{H}/\hspace{-.15cm}\sim$, consider the set $([\hs{h}],[\hs{l}])=\{[\hs{k}]|[\hs{h}]\leq [\hs{k}]\leq [\hs{l}]\}$. Let $\hs{h}\leq \hs{l}$ be a fixed pair of representatives of $[\hs{h}]\leq [\hs{l}]$ that are \inseparable{($[\hs{h}]\cup [\hs{l}]$)}, by \Obs \ref{unfolding2}, any element in $([\hs{h}],[\hs{l}])$ must have a representative that lies in $(\hs{h},\hs{l})=\{\hs{k}|\hs{h}\leq \hs{k}\leq \hs{l}\}$, thus, by local finiteness of $(\Hs{H},\leq )$, the set $([\hs{h}],[\hs{l}])$ is finite. 
\end{proof}

As we have seen in Example \ref{example of quotient}, hyperplanes in the quotient can be transverse even if their equivalence classes have no transverse hyperplanes. The next lemma explains when this happens.

\begin{lemma} \label{preimage of crossing hyps}
Let $(\Hs{H},\leq ,\comp{})$ be a pocset, and let $\sim$ be an admissible equivalence relation on $\Hs{H}$.
Assume that $[\hs{h}]$ and $[\hs{k}]$ are transverse in $\Hs{H}\modsim$. Then one of the following happens.
\begin{itemize}
\item There exists $\hs{h} \in [\hs{h}]$ and $\hs{k} \in [\hs{k}]$ that are transverse.
\item There exists $\hs{h}_1$ and $\hs{h}_2$ in $[\hs{h}]$ and $\hs{k} \in [\hs{k}]$ such that $\hyp{k}$ separates $\hyp{h}_1$ and $\hyp{h}_2$.
\item There exists $\hs{k}_1$ and $\hs{k}_2$ in $[\hs{k}]$ and $\hs{h} \in [\hs{h}]$ such that $\hyp{h}$ separates $\hyp{k}_1$ and $\hyp{k}_2$.
\end{itemize}
\end{lemma}

\begin{proof}
Let us assume that $[\hs{h}],[\hs{k}]$ do not satisfy any of the cases of the lemma, and prove that $[\hyp{h}]\not\pitchfork[\hyp{k}]$.
Let $\Hs{I}$ (resp. $\Hs{J}$) be the the set of elements in $[\hs{h}]$ (resp. $[\hs{k}]$) that are \inseparable{$[\hyp{h}] \cup [\hyp{k}]$} from an element in $[\hs{k}]$ (resp. $[\hs{h}]$). 
By assumption, the hyperplanes in $\Hyp{I}\cup\Hyp{J}$ are not transverse.
Note that the arrangement of $[\hs{h}]$ and $[\hs{k}]$ in $\Hs{H}\modsim$ only depends on the orientations of elements of $\Hs{I} \cup \Hs{J}$.
Let us first prove that $\Hyp{I} \cup \Hyp{J}$ is a facing collection of hyperplanes. Assume otherwise that a hyperplane $\hyp{l}_1$ separates $\hyp{l}_2$ and $\hyp{l}_3$ in $\Hyp{I} \cup \Hyp{J}$. We are in one of the following case.
\begin{itemize}
\item $[\hyp{l}_2]=[\hyp{l}_3]$. We can assume that $\hyp{l}_2$ and $\hyp{l}_3$ belong to $\Hyp{I}$. Then since no element of $[\hyp{k}]$ separates $\Hyp{I}$, we must have $\hyp{l}_1 \in \Hyp{I}$. Now both $\hyp{l}_2$ and $\hyp{l}_3$ are \inseparable{$[\hyp{h}] \cup [\hyp{k}]$} with elements $\hyp{l}'_2$ and $\hyp{l}'_3$ in $\Hyp{J}$. The inseparability implies that $\hyp{l}_1$ separates $\hyp{l}'_2$ and $\hyp{l}'_3$, a contradiction.
\item $[\hyp{l}_2]\neq[\hyp{l}_3]$. We can assume that $\hyp{l}_1$ and $\hyp{l}_2$ belong to $\Hyp{I}$ and $\hyp{l}_3$ belongs to $\Hyp{J}$. As before, we can find $\hyp{l}'_2$ in $\Hs{J}$ such that $\hyp{l}_1$ separates $\hyp{l}'_2$ and $\hyp{l}_3$, again a contradiction.
\end{itemize}

Now, applying property \ref{welloriented} on the elements of $\Hyp{I} \cup \Hyp{J}$, elements of $\Hs{I}$ are either all facing towards elements of $\Hyp{J}$ or all facing away from $\Hyp{J}$. The same happens for $\Hs{J}$. In all the cases, by definition of the pocset structure on a quotient, $[\hs{h}]$ and $[\hs{k}]$ are not transverse.
\end{proof}

Note that the converse to Lemma \ref{preimage of crossing hyps} does not hold in general (\emph{e.g.} the red and light blue hyperplanes in Example \ref{example of quotient}).


\section{Maps between pocsets}

In what follows, quotients will arise from maps between pocsets.
For classical Stallings' $G$-tree folding sequences the maps considered are simplicial ($G$-equivariant) maps of trees. 
The analogous notion in our setting will be called \emph{resolutions}.
Similarly, the analogous notion of injective simplicial maps of trees -- $L^1$ isometric embeddings -- will be simply called \emph{embeddings}.
The next definitions make these notion precise in the language of pocsets, using the notion of \emph{admissible maps}.

\begin{definition}
Let $(\Hs{H},\leq ,\comp{})$ and $(\Hs{H}',\leq ,\comp{})$ be pocsets. A map $f:\Hs{H}\to\Hs{H}'$ is  an \emph{admissible map} of pocsets if the following hold:
\begin{enumerate}[label=(AM\arabic*)]
	\item \label{AMcompinv}for all $\hs{h}\in\Hs{H}$, $f(\comp{\hs{h}})=\comp{f(\hs{h})}$,  \footnote{Property \ref {AMcompinv} shows that $f$ induces a well-defined map on hyperplanes, which we denote as well by $f:\Hyp{H} \to \Hyp{H}'$ }
	\item \label{AMcrossinghyps}for all $\hyp{h}\pitchfork\hyp{k}\in\Hs{H}$, $f(\hyp{h})\pitchfork f(\hyp{k})$,
	\item \label{AMfacinghyps}for all $\hs{h},\hs{k}\in\Hs{H}$ facing halfspaces, that satisfy $f(\hyp{h})=f(\hyp{k})$ and that are \inseparable{$f^{-1}(f(\hyp{h}))$}, we have $f(\hs{h})=f(\hs{k})$, and
	\item \label{AMcompactible_image} every $\hyp{h}'\in\Hyp{H}'\setminus f(\Hyp{H})$ has an orientation $\hs{h}'$ that is compatible with all the halfspaces in $f(\Hs{H})$.
\end{enumerate}

An admissible map $f:\Hs{H}\to\Hs{H}'$ is an \emph{embedding of pocsets} if $f$ is injective and for all $\hs{h},\hs{k}\in\Hs{H}$, if $f(\hs{h})\leq f(\hs{k})$ then $\hs{h}\leq \hs{k}$. We denote such a map by $\Hs{H}\embedsin\Hs{H}'$. Note that by injectivity property \ref{AMfacinghyps} is superfluous in this case.

An admissible map $f:\Hs{H}\to\Hs{H}'$ is a \emph{resolution of pocsets} if $f$ satisfies that the map $\Hs{H}/{\sim_f}\to \Hs{H}'$ is an embedding of pocsets, where $\sim_f$ is the admissible equivalence relation defined by $\hs{h}\sim_f \hs{k}$ if $f(\hs{h})=f(\hs{k})$.
\end{definition}

\begin{remark}
The quotient map $f:\Hs{H}\to\Hs{H}/\hspace{-.15cm}\sim$ for an admissible equivalence relation $\sim$ is a resolution.
\end{remark}

The motivating example for this definition is the geometric resolution of an action of a finitely presented group on a \CCC, as the following lemma shows. The proof of the lemma is straight forward from the definition of the geometric resolution.

\begin{lemma}\label{previouspaper}
Let $G$ be a finitely presented group, and let $\CC{X}'$ be a finite dimensional CAT(0) cube complex on which $G$ acts. Let $\CC{X}$ be the geometric resolution of $\CC{X}'$ described in \cite{BeLa15}. Then the map $f:\Hs{H}(\CC{X})\to\Hs{H}(\CC{X}')$ is a resolution of pocsets.\qed
\end{lemma} 


The following lemma describes how resolutions between pocsets can be realized as maps between the associated CAT(0) cube complexes.

\begin{lemma} \label{AMtoCCCmap}
Let $f:\Hs{H}\to\Hs{H}'$ be a resolution, and let $\CC{X},\CC{X}'$ be the CAT(0) cube complexes associated with $\Hs{H},\Hs{H}'$ respectively. 
There is a CAT(0) cube subcomplex $\CC{Z}\subseteq\CC{X}'$ that decomposes as a product $\CC{Z}=\CC{Z}_1\times\CC{Z}_2$, such that $f(\Hs{H})=\Hs{H}(\CC{Z}_1)$, and the map $f$ induces a canonical combinatorial (and hence $L_1$-distance-non-increasing) map $F:\CC{X}\to\CC{Z}_1$. 
In particular, for every choice of vertex $\CCv{z}\in\CC{Z}_2$ the map $f$ induces a map $F:\CC{X}\to\CC{Z}_1\times\{\CCv{z}\}\subseteq\CC{Z}\subseteq\CC{X}'$.

If moreover, $f$ is an embedding of pocsets, the induced map $F$ is an $L_1$-embedding.
\end{lemma}

\begin{proof}
We partition the set $\Hyp{H}'$ into three subsets in the following way: let $\Hyp{H}_1=f(\Hyp{H})$; let  $\Hyp{H}_2$ be the set of all hyperplanes in $\Hyp{H}'\setminus f(\Hyp{H})$ that are transverse to all hyperplanes in $f(\Hyp{H})$; and let $\Hyp{H}_3$ be the remaining set, i.e $\Hyp{H}_3=\Hyp{H}'\setminus(\Hyp{H}_1\cup\Hyp{H}_2)$.

By property \ref{AMcompactible_image} and the definition, every hyperplane $\hyp{h}'\in\Hyp{H}_3$ has a unique choice of halfspace $\hs{h}'$ that contains or is transverse to any hyperplane in $\Hyp{H}_1$. By the definition of $\Hyp{H}_2$ the same choice of $\hs{h}'\in\Hs{H}_3$ will either contain or be transverse any hyperplane in $\Hyp{H}_2$.

Thus the subcomplex $\CC{Z}=\bigcap_{\hyp{h}'\in\Hs{H}_3} \hs{h}'$, where $\hs{h}'$ is the unique choice of halfspace that satisfies the above, is isomorphic to $\CC{X}(\Hs{H}_1\cup\Hs{H}_2)$, which naturally decomposes as a product $\CC{Z}=\CC{Z}_1\times\CC{Z}_2$ where $\CC{Z}_i=\CC{X}(\Hs{H}_i)$ for $i=1,2$.

We define the map $F:\CC{X}\to\CC{Z}_1$ in the following way. For a vertex $\CCv{x}$ of $\CC{X}$, which we think of as the ultrafilter choice function $\CCv{x}:\Hyp{H}\to\Hs{H}$, we define $F(\CCv{x})$ to be the following ultrafilter. For all $\hyp{h}'\in\Hyp{H}_1$, let $F(\CCv{x})(\hyp{h}')$ be $f(\CCv{x}(\hyp{h}))$ where $\hyp{h}$ is a hyperplane of $f^{-1}(\hyp{h}')$ such that $\CCv{x}(\hyp{h})$ is minimal in $\CCv{x}$. This is well defined by the axiom \ref{AMfacinghyps}.

The function $F(\CCv{x}):\Hyp{H}_1\to\Hs{H}_1$ is an ultrafilter. 
Let $\hyp{h}'$ and $\hyp{k}'$ be distinct hyperplanes, we have to show that $F(\CCv{x})(\hyp{h}')$ and $F(\CCv{x})(\hyp{k}')$ are compatible. Assume by contradiction that they are incompatible.
Then, their pre-images in $\Hs{H}/{\sim_f}$ under the embedding of pocsets $\Hs{H}/{\sim_f}\embedsin \Hs{H}'$ are incompatible.
Let $\hs{h},\hs{k}$ be the minimal halfspaces $\CCv{x}(\hyp{h}),\CCv{x}(\hyp{k})$ associated to hyperplanes $\hyp{h}$ and $\hyp{k}$ in $f^{-1}(\hyp{h'})$ and $f^{-1}(\hyp{k'})$ respectively, as described in the definition of $F(\CCv{x})$ above. 
Then they satisfy one of the following.
\begin{itemize}
\item The hyperplane $\hyp{k}$ separates $\CCv{x}$ and $\hyp{h}$, which in particular implies that the halfspace $\hs{h}$ contains $\hyp{k}$, and that $\hyp{h},\hyp{k}$ are \inseparable{$[\hyp{h}]$}. 
But since $[\hs{h}]<[\comp{\hs{k}}]$, by \Obs \ref{observation}, we get that $\hyp{k}\subset \comp{\hs{h}}$, a contradiction.
\item The hyperplane $\hyp{h}$ separates $\CCv{x}$ and $\hyp{k}$, which similarly gives a contradiction. 
\item The two halfspaces $\hs{h}$ and $\hs{k}$ are facing. We may assume that they are \inseparable{($[\hyp{h}]\cup[\hyp{k}]$)} (where the equivalence class is with respect to $\sim_f$), since otherwise we can find representative in $\CCv{x}$ satisfying the first bullet point. Thus, $\hs{h}$ and $\hs{k}$ must be incompatible, contradicting the fact that they both contain $\CCv{x}$.
\end{itemize}
Finally, the ultrafilter $F(\CCv{x})$ is DCC since $\CCv{x}$ is. 

This defines a map on vertices.  To show that this map extends to edges, it is enough to show that adjacent vertices are sent to adjacent vertices. Recall that two vertices are adjacent if their ultrafilters differ on exactly one hyperplane. Each of the two orientation of this hyperplane is minimal in the corresponding vertex. Hence their images have to differ exactly on this hyperplane by the construction of the map $F$. Similarly, the map extends to higher dimensional cubes because any pairwise transverse set of distinct hyperplanes projects injectively to a pairwise transverse set of distinct hyperplanes, by property \ref{AMcrossinghyps}.


If moreover $f$ is an embedding of pocsets, using \ref{AMcompactible_image} the collection of hyperplanes that separate $\CCv{x}$ and $\CCv{y}$ is in one-to-one correspondence with the collection of hyperplanes that separate $F(\CCv{x}')$ and $F(\CCv{y}')$. Thus, $F$ is an $L_1$-embedding.
\end{proof}

\section{Stallings' folds}

We say that a group acts \emph{without inversion} on a \CCC if there are no elements that send a halfspace to its complement. 
Given a group acting on a \CCC, by replacing the \CCC by its cubical barycentric subdivision we may assume that the action is without inversions.
Therefore, in what follows we consider only group actions without inversions.

The main goal of constructing Stallings' folds is to prove that, under some conditions, a resolution can be decomposed as a finite sequence of simpler quotients, called \emph{elementary folds}, which we introduce in the following definition.

\begin{definition}

Let a group $G$ act on the pocsets $\Hs{H},\Hs{H}'$, and let $f:\Hs{H}\to\Hs{H}'$ be a $G$-equivariant resolution of pocsets. 
Two facing halfspaces $\hs{h}_1$ and $\hs{h}_2$ of $\Hs{H}$ are \emph{\elemfoldable} if $\hs{h}_1\sim_f \hs{h}_2$, that is $f(\hs{h}_1)=f(\hs{h}_2)$, and there are no facing pairs $\hs{k}_1\sim_f \hs{k}_2$ that satisfy $\hs{k}_i\le\hs{h}_i$ for $i=1,2$.  
Equivalently, the pair $\hyp{h}_1,\hyp{h}_2$ is \inseparable{$[\hyp{h}_1]_{\sim_f}$}, and there are no pairs of $\sim_f$-equivalent hyperplanes that both separate $\hyp{h}_1,\hyp{h}_2$.

An \emph{elementary fold} is a quotient of the form $\Hs{H}/\hspace{-.15cm}\sim$ where $\sim$ is the minimal $G$-invariant and $\comp{}$-invariant equivalence relation generated by identifying an \elemfoldable pair $\hs{h}_1\sim\hs{h}_2$. 
We will denote the quotient map of an elementary fold by $\phi:\Hs{H}\leadsto\Hs{H}/\hspace{-.15cm}\sim$.
\end{definition}

\begin{remark}\label{rmk: an elementary fold exists}
	If $f:\Hs{H}\to\Hs{H}'$ is a resolution that is not an embedding, then $f$ admits two hyperplanes that are identified by $f$. Between them there are only finitely many other pairs that are identified by $f$ thus a minimal such pair is an elementary foldable pair.
\end{remark}

It is worth noting at this point that classical Stallings' folds for $G$-trees are indeed elementary folds also in our setting.

As we said, the goal is to show that certain resolutions can be factored by a finite sequence of elementary folds; this is the content of Lemma \ref{one factorization} and Proposition \ref{finite_folding}.
Before diving into more technical lemmas, we demonstrate the basic principles of these lemmas in the following example. 

\begin{example}\label{example of folding sequence}
Let us consider the action of $\Z$ on the line by translations by multiples of $2$, and the action of $\Z$ on the cubical barycentric subdivision of a square by rotations by multiples of $\pi/2$ (See top and bottom figures in Figure \ref{fan}) 

As pocsets, the former is $\Hs{H}=\{\hs{h}_i,\comp{\hs{h}}_i | i\in\Z\}$ with the poset structure $\hs{h}_i\ge\hs{h}_j$ and $\comp{\hs{h}}_i\le \comp{\hs{h}}_j$ for all $i\le j$, and the obvious complementation involution. 
The latter is $\Hs{H}=\{\hs{k}_i,\comp{\hs{k}}_i | i\in\Z/4\Z \}$ with the poset structure $\comp{\hs{k}}_i\le\hs{k}_{i+2}$ for all $i\in\Z/4\Z$.
The action of $\Z=<a>$ on $\Hs{H}$ is given by $ a \hs{h}_i = \hs{h}_{i+2}$ (and $a \comp{\hs{h}}_i = \comp{\hs{h}}_{i+2}$), and the action of $\Z$ on $\Hs{H}'$ is given by $ a \hs{k}_i = \hs{k}_{i+1}$ (and $a \comp{\hs{k}}_i = \comp{\hs{k}}_{i+1}$).
We consider the $\Z$-equivariant map $f:\Hs{H}\to\Hs{H}'$ that is defined  on $\hs{h}_i$ by 
$$
f(\hs{h}_i)=
\begin{cases}
	\hs{k}_{i/2\;(\rm{mod}\;4)} & i\equiv 0\;(\rm{mod}\; 2)\\ 
	\comp{\hs{k}}_{(i+3)/2\;(\rm{mod}\;4)} & i\equiv1\;(\rm{mod}\; 2)\\ 
\end{cases}
$$
(and on $\comp{\hs{h}}_i$  by $f(\comp{\hs{h}}_i) = \comp{f(\hs{h}_i)}$)

The $\Z$-invariant equivalence relation generated by $\hs{h}_{-1}\sim\comp{\hs{h}}_2 $ is an elementary fold.
After folding, we obtain a fan shaped square complex, in which $\Z$-many squares share a common vertex and two consecutive squares share an edge, on which $\Z$ acts by fixing the shared vertex and shifting the squares (see the middle figure in Figure \ref{fan}). As a pocset, it is isomorphic to $\Hs{H}_1=\{\hs{t}_i,\comp{\hs{t}}_i | i\in \Z\}$, with the poset $\comp{\hs{t}}_i\le\hs{t}_{j}$ for all $|i-j|\ge 2$, the $\Z$ action is given by $a \hs{t}_i = \hs{t}_{i+1}$ and the folding map $\phi_0$ is given by 

$$
\phi_0(\hs{h}_i)=
\begin{cases}
\hs{t}_{i/2} & i\equiv 0\;(\rm{mod}\; 2)\\ 
\comp{\hs{t}}_{(i+3)/2} & i\equiv1\;(\rm{mod}\; 2)\\ 
\end{cases}
$$

The map $f$ induces a map $f_1:\Hs{H}_1\to\Hs{H}'$ which can be written explicitly by $f_1 (\hs{t}_i) = \hs{k}_{i\;(\rm{mod}\;4)}$. 
This map is again a resolution, and the $\Z$-equivariant equivalence relation generated by $\hs{t}_0\sim \hs{t}_4$ is an elementary fold. The resulting quotient $\Hs{H}_2$ is isomorphic to the pocset $\Hs{H}'$, and under this isomorphism the quotient map $\phi_1:\Hs{H}_1\to\Hs{H}_2=\Hs{H}'$ is the map $f_1$. 
Thus, our sequence of folds $$\Hs{H}=\Hs{H}_0 \overset{\phi_0}{\leadsto} \Hs{H}_1 \overset{\phi_1}{\leadsto} \Hs{H}_2 \embedsin \Hs{H}'$$ terminated with the pocset $\Hs{H}_2$ which embeds in  $\Hs{H}'$ (in this case, they are isomorphic).

\begin{figure}[t]
\begin{center}
\def\svgwidth{0.7\textwidth}
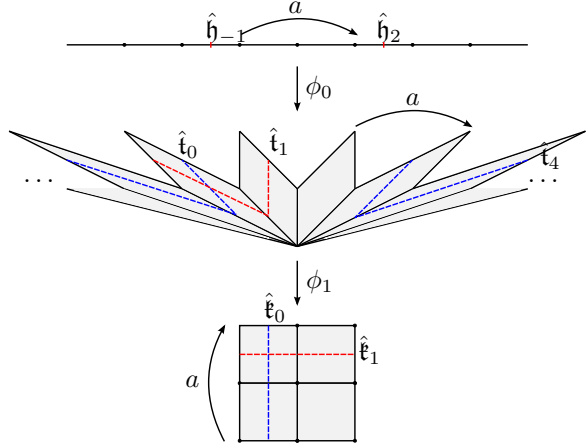
\caption{A folding sequence for an action of $\Z$.}
\label{fan}
\end{center}
\end{figure} 
\end{example}

Our main immediate goal is Lemma \ref{one factorization}, which states that a $G$-equivariant resolution can be factored through an elementary fold.
However, we will first need some technical lemmas that describe what the equivalence relation of an elementary fold looks like, and how to relate it to the arrangement of the hyperplanes in its quotient.


The following lemma shows that all pairs of hyperplanes in an equivalence class of an elementary fold are facing.
In Example \ref{example of folding sequence}, each of the classes of the first fold consists of two hyperplanes, and the classes of the second fold are infinite facing collections hyperplanes. 

\begin{lemma}\label{orientation}
Let $f:\Hs{H}\to\Hs{H}'$ be a $\gp{G}$-equivariant resolution of pocsets, and let $\hs{h}_1,\hs{h}_2$ be an elementary foldable pair of halfspaces. Let $\sim$ be the relation defining the elementary fold of $\hs{h}_1\sim\hs{h}_2$. Then for any hyperplane $\hyp{h}\in\Hyp{H}$, there exists an orientation $[\hs{h}]$ of $[\hyp{h}]$ such that the halfspaces in $[\hs{h}]$ are facing. In particular, the collection $[\hyp{h}]$ is facing.
\end{lemma}

\begin{proof}
Let $\hs{h}_1,\hs{h}_2$ be the orientation of $\hyp{h}_1,\hyp{h}_2$. 
By definition $\hs{h}_1,\hs{h}_2$ are facing.
Extend the orientation using the action of $G$ to the orbits $G.\hyp{h}_1$ and $G.\hyp{h}_2$, this is well defined since $G$ acts without inversions.
For any two equivalent hyperplanes $\hyp{t}_1\sim\hyp{t}_2$ there exists a sequence of distinct hyperplanes $\hyp{t}_1=\hyp{l}_1, \hyp{l}_2, \dots, \hyp{l}_n = \hyp{t}_2$, such that for all $i$ there exists $g_i\in \gp{G}$  such that $(\hyp{l}_i, \hyp{l}_{i+1}) = (g_i\cdot\hyp{h}_1,g_i\cdot\hyp{h}_2)$ or $(g_i\cdot\hyp{h}_2,g_i\cdot\hyp{h}_1)$.  Since the set of \elemfoldable pairs is stable by the action, all the pairs are \elemfoldable and none of the hyperplanes $\hyp{l}_i$ separate any of the pairs $(\hyp{l}_j,\hyp{l}_{j+1})$. In addition, by property \ref{AMcrossinghyps}, $\hyp{l}_i$ is not transverse to $\hyp{l}_j$ for all $i,j$. This implies that for all $i\ne j$, $\hyp{l}_j \subset \hs{l}_i$, which proves the lemma.
\end{proof}

%

\begin{lemma}\label{lem: folds are AER}
Let $f_0:\Hs{H}\to\Hs{H}'$ be a $G$-equivariant resolution. Let $\sim$ be an elementary fold as above. Then $\sim$ is an admissible equivalence relation.
\end{lemma}
\begin{proof}
By definition, the equivalence relation $\sim$ satisfies Property \ref{compinv}. Properties \ref{selfcomp} and \ref{welloriented} follows directly from Lemma \ref{orientation}.
Property \ref{selfcross} follows from the fact that the map $f_0$ is admissible, and thus if $\hyp{h}\pitchfork\hyp{k}$ then their images in $\Hs{H}'$ are distinct and in particular they are not $\sim$-equivalent.
\end{proof}

The following lemma characterizes when two hyperplanes are transverse after an elementary fold. 
It might be useful to compare the two cases of the lemma with Example \ref{example of folding sequence}. The transverse pairs of hyperplanes after the first fold correspond to Case \ref{case2}, while those of the second fold correspond to Case \ref{case1}.
%
%

\begin{observation}\label{unfolding1}
Let $f:\Hs{H}\to\Hs{H}'$ be a $G$-equivariant resolution of pocsets.
Let $\phi:\Hs{H}\leadsto\Hs{H}/\hspace{-.15cm}\sim $ be an elementary fold of $\Hs{H}$ defined by folding the \elemfoldable pair $\hs{h}_1\sim\hs{h}_2$. If $[\hyp{h}]\pitchfork[\hyp{k}]$ in $\Hyp{H}/\hspace{-.15cm}\sim $ then up to interchanging $[\hyp{h}]$ and $[\hyp{k}]$, there exists an orientation $[\hs{h}]$ and $[\hs{k}]$ of $[\hyp{h}]$ and $[\hyp{k}]$ such that one of the following happens: 
\begin{enumerate}
\item\label{case1} there exist $\hs{h}\in [\hs{h}]$ and $\hs{k}\in [\hs{k}]$ such that $\hs{h}\pitchfork\hs{k}$;
\item\label{case2}   there exist an element $g\in G$ and $\hs{k}\in [\hs{k}]$ such that $[g\cdot\hs{h}_1] = [g\cdot\hs{h}_2] = [\hs{h}]$ and $g\cdot \comp{\hs{h}}_1 < \hs{k} < g\cdot \hs{h}_2$ (up to interchanging $\hs{h}_1$ and $\hs{h}_2$).
\end{enumerate}

Conversely, if there exists $\hyp{h}$, $\hyp{h}_\bullet$ and $\hyp{k}$ in $\Hyp{H}$ such that $\hyp{h} \sim \hyp{h}_\bullet$ and $\hyp{k}$ separates $\hyp{h}$ and $\hyp{h}_\bullet$ then $[\hyp{h}] \pitchfork [\hyp{k}]$ in $\Hyp{H}/\hspace{-.15cm}\sim $.
\end{observation}

\begin{proof}
Assume that $[\hs{h}]\pitchfork[\hs{k}]$ and Case \ref{case1} does not hold. From Lemma \ref{preimage of crossing hyps}, without loss of generality, assume that an element $\hyp{k}\in [\hyp{k}]$ separates $[\hyp{h}]$.
By Lemma \ref{orientation}, all the elements of $[\hs{h}]$ are facing, we can choose two elements $\hs{h},\hs{h}_\bullet$ such that $\comp{\hs{h}} < \hs{k} < \hs{h}_\bullet$.
Using the same construction as in the proof of Lemma \ref{orientation}, there exists a sequence of distinct hyperplanes $\hyp{h}=\hyp{l}_1, \hyp{l}_2, \dots, \hyp{l}_n = \hyp{h}_\bullet$, such that for all $i$ there exists $g_i\in \gp{G}$  such that $(\hyp{l}_i, \hyp{l}_{i+1}) = (g_i\cdot\hyp{h}_1,g_i\cdot\hyp{h}_2)$ or $(g_i\cdot\hyp{h}_2,g_i\cdot\hyp{h}_1)$.
 Since no element of $[\hs{k}]$ is transverse to an element of $[\hs{h}]$, there exists $i$ such that $g_i\cdot\comp{\hs{h}}_1 < \hs{k} < g_i\cdot\hs{h}_2$ (up to interchanging $\hs{h}_1$ and $\hs{h}_2$). 
 
 For the last part of the lemma, let us first remark that from Lemma \ref{orientation} we know that $\hyp{h} \not \sim \hyp{k}$. If $\hyp{k}$ is transverse to an element of $[\hyp{h}]$ then by definition $[\hyp{h}]\pitchfork[\hyp{k}]$. So we may assume that $\hyp{k}$ is disjoint from the elements of $[\hyp{h}]$. Using the same construction as previously, there exists $g \in G$ such that $\hyp{h}\sim g\cdot\hyp{h}_1$ and $\hyp{k}$ separates $g\cdot \hyp{h}_1$ and $g\cdot \hyp{h}_2$.
Since $\hyp{h}_1$ and $\hyp{h}_2$ is an \elemfoldable pair, the hyperplane $\hyp k$ cannot be equivalent to any other hyperplane separating  $g\cdot \hyp{h}_1$ and $g\cdot \hyp{h}_2$, that is, both the pairs $\hyp k$, $g\cdot \hyp{h}_1$ and $\hyp{k}$, $g\cdot \hyp{h}_2$ are \inseparable{$[\hyp{h}] \cup [\hyp{k}]$}, which by the definition of the pocset structure of the quotient implies that $[\hyp{h}] \pitchfork [\hyp{k}]$.
 \end{proof}



\begin{observation}\label{no more crossing}
Let $f:\Hs{H} \rightarrow \Hs{H}'$ be a $G$-equivariant resolution, and $\phi$ be an elementary fold of $f$. Let $\hs{h}$ and $\hs{k}$ be two halfspaces of $\Hs{H}$. If $\phi(\hs{h})$ and $\phi(\hs{k})$ are transverse, then so are $[\hs{h}]_{\sim_f}$ and $[\hs{k}]_{\sim_f}$ in $\Hs{H}/\hspace{-.15cm}\sim_f$. In particular their images $f(\hs{h})$ and $f(\hs{k})$ are transverse in $\Hs{H}'$.
\end{observation}

\begin{proof}
 Let $\hs{h}$ and $\hs{k}$ be two halfspaces of $\Hs{H}$ and assume $\phi(\hs{h})$ and $\phi(\hs{k})$ are transverse.
By \Obs \ref{unfolding1}, either some preimages $\hs{h}_\bullet$ and $\hs{k}_\bullet$ are transverse in $\Hs{H}$, in which case, by Property \ref{AMcrossinghyps} their images are transverse in $\Hs{H}'$. 
Or, up to swapping $\hs{h}$ and $\hs{k}$ there exists $\hs{k}_\bullet\sim \hs{k}$ and $g\in G$ such that $g\cdot\comp{\hs{h}}_1 < \hs{k}_\bullet < g\cdot \hs{h}_2$ and $g\cdot\hs{h}_1\sim g\cdot \hs{h}_2 \sim \hs{h}$ where $\phi$ is the elementary fold that is generated by the elementary foldable pair $\hs{h}_1\sim\hs{h}_2$.
Since $\hs{h}_1$ and $\hs{h}_2$ are an \elemfoldable pair, the only preimage of $f(\hyp{k})$ separating  $g\cdot\hyp{h}_1$ and $g\cdot \hyp{h}_2$ is $\hs{k}_\bullet$.
Therefore, by the definition of the order, the images of $\hs{h}$ and $\hs{k}$ in $\Hs{H}/\hspace{-.15cm}\sim_f$ are transverse, and thus $f(\hyp{k})$ and $f(\hyp{h})$ are also transverse.
\end{proof}

In the setting of trees, if two edges $\hyp{h}$ and $\hyp{l}$ are separated by a third edge $\hyp{k}$, and after a fold $\phi$ the edge $\phi(\hyp{k})$ does not separate $\phi(\hyp{h})$ and $\phi(\hyp{l})$, then either $\phi$ folds $\hyp{k}$ to one of $\hyp{h},\hyp{l}$, or $\phi$ folds $\hyp{k}$ to another edge $\hyp{k}'$ that also separates $\hyp{h}$ and $\hyp{l}$, and the triple $\phi(\hyp{h})$,  $\phi(\hyp{l})$ and  $\phi(\hyp{k})$ is a facing triple. 
See Figure \ref{fig: reunited at last}.

The following lemma describes a similar behavior of CAT(0) cube complex folds.
Again, it is worth comparing also to Example \ref{example of folding sequence}, where after the first fold there is no pair of hyperplanes which is separated by a third.


\begin{figure}[t]
	\begin{center}
		\def\svgwidth{0.7\textwidth}
		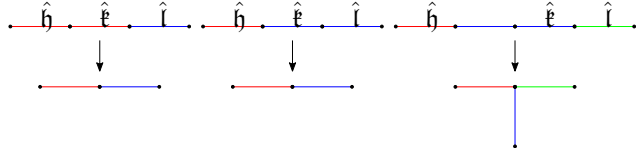
		\caption{Hyperplanes that were separated can be reunited thanks to Stallings.}
		\label{fig: reunited at last}

	\end{center}
\end{figure}

\begin{lemma}\label{thirdguy}

Let $f_0:\Hs{H}\to\Hs{H}'$ be a $G$-equivariant resolution and $\sim$ be an elementary fold as above. Let $\phi:\Hs{H}\to \Hs{H}/\hspace{-.15cm}\sim $ be the map associated to the fold.

Let $\hyp{h}$, $\hyp{k}$ and $\hyp{l} \in \Hs{H}$ such that $\hyp{k}$ separates $\hyp{h}$ and $\hyp{l}$. If their images by $f_0$ are not transverse and their images by $\phi$ are such that $[\hyp{k}]$ does not separate $[\hyp{h}]$ and $[\hyp{l}]$ then one of the following happens:
\begin{enumerate}
\item  $\hyp{h}\sim \hyp{k}$,
\item   $\hyp{l}\sim \hyp{k}$, or
\item \label{thirdcaselemmathirdguy} the hyperplanes $[\hyp{h}]$, $[\hyp{k}]$ and $[\hyp{l}]$ form a facing triple. 
\end{enumerate}
Moreover, in case (\ref{thirdcaselemmathirdguy}) 
there are exactly two elements of $[\hyp{k}]$ separating $\hyp{h}$ and $\hyp{l}$.
\end{lemma}

\begin{proof}
Assume that $[\hyp{h}] \neq [\hyp{k}]$ and $[\hyp{l}] \neq [\hyp{k}]$. Since $[\hyp{k}]$ does not separates $[\hyp{h}]$ and $[\hyp{l}]$, if the three hyperplanes do not form a facing triple then $[\hyp{h}]$ or $[\hyp{l}]$ separates the other two. Without loss of generality we can assume that $[\hyp{h}]$ separates $[\hyp{k}]$ and $[\hyp{l}]$. By Lemma \ref{unfolding2}, there would be a representative of $[\hyp{h}]$ in between $\hyp{k}$ and $\hyp{l}$. But by the converse part of Lemma \ref{unfolding1}, the hyperplanes $[\hyp{h}]$ and $[\hyp{k}]$ would be transverse, which by Lemma \ref{no more crossing} is a contradiction.

To prove the last part of the lemma, notice that there are at least two lifts of $[\hyp{k}']$ in between $\hyp{h}$ and $\hyp{k}$. Indeed under the orientation such that all three halfspaces are facing, every inseparable lift of pairs should be facing, but this implies that there is more than one lift of $[\hyp{k'}]$ in between $\hyp{h}$ and $\hyp{l}$.
Lemma \ref{orientation} shows that there are exactly two.
\end{proof}

\section{Folding sequences}

Let us now state the key lemma in proving that resolutions can be decomposed as sequences of folds.
\begin{lemma}\label{one factorization}
Let $f_0:\Hs{H}\to\Hs{H}'$ be a $G$-equivariant resolution. Let $\sim$ be an elementary fold and $\phi:\Hs{H}\leadsto \Hs{H}/\hspace{-.15cm}\sim $ be the map associated to the fold as above. Then the map $f_1:\Hs{H}/\hspace{-.15cm}\sim\to\Hs{H}'$ is admissible and $(\Hs{H}/\hspace{-.15cm}\sim)/\hspace{-.15cm}\sim_{f_1} =\Hs{H}/\hspace{-.15cm}\sim_{f_0}$ as pocsets. In particular $f_1$ is a resolution.
\end{lemma}

\begin{proof}
To simplify the notation, we denote $\Hs{H}_\sim$, $\Hs{H}_0$ and $\Hs{H}_1$ instead of $\Hs{H}/\hspace{-.15cm}\sim$, $\Hs{H}/\hspace{-.15cm}\sim_{f_0}$ and  $(\Hs{H}/\hspace{-.15cm}\sim)/\hspace{-.15cm}\sim_{f_1}$.
Elements in $\Hs{H}_\sim$, $\Hs{H}_0$ and $\Hs{H}_1$ will be denoted with indices $\cdot_\sim$, $\cdot_0$ and $\cdot_1$.

We first show that $f_1$ is admissible. 
\begin{itemize}
\item Properties \ref{AMcompinv} and \ref{AMcompactible_image} clearly follow from those of $f_0$.
\item Property \ref{AMcrossinghyps} is given by \Obs \ref{no more crossing}
\item Property \ref{AMfacinghyps}. 
Let $\hs{h}_\sim$ and $\hs{k}_\sim$ be facing halfspaces in $\Hs{H}_\sim$ and assume that $f_1(\hyp{h}_\sim) = f_1(\hyp{k}_\sim)$ and that they are \inseparable{$f_1^{-1}(f_1(\hyp{h}_\sim))$}.

Let $\hs{h}\in \phi^{-1}(\hs{h}_\sim)$ and $\hs{k}\in \phi^{-1}(\hs{k}_\sim)$ be \inseparable{$\phi^{-1}(\hyp{h}_\sim)\cup \phi^{-1}(\hyp{k}_\sim)$} hyperplanes. No element of $f_0^{-1}(f_0(\hyp{h}))$  in between $\hyp{h}$ and $\hyp{k}$ is $\sim$-equiva\-lent to $\hyp{h}$ or $\hyp{k}$. 
Since $\hyp{h}_\sim$ and $\hyp{k}_\sim$ are \inseparable{$f_1^{-1}(f_1(\hyp{h}_\sim))$}, they are not separated by the images under $\phi$ of the elements of $f_0^{-1}(f_0(\hyp{h}))$ that separate $\hyp{h}$ and $\hyp{k}$ in $\Hs{H}$. 
Thus by Lemma \ref{thirdguy} the  number of elements in $f_0^{-1}(f_0(\hyp{h}))$ in between $\hyp{h}$ and $\hyp{k}$ is even because they come in pairs of $\sim$ equivalent hyperplanes. Now  using Property \ref{AMfacinghyps} of $f_0$, we obtain that these elements form an alternating sequence of facing and incompatible pairs, therefore $f_0(\hs{h}) = f_0(\hs{k})$, hence $f_1(\hs{h}_\sim) = f_1(\hs{k}_\sim)$.
\end{itemize}

We are left to show that $\Hs{H}_0=\Hs{H}_1$. By definition there is a bijection between the two sets, and it is easy to see that it commutes with $\comp{}$. We need to show that the pocset structure is the same. That is, given two elements $\hs{h}$ and $\hs{k}$ in $\Hs{H}$,
\begin{itemize}
\item if $\hs{h}_0$ and $\hs{k}_0$ are transverse then  $\hs{h}_1$ and $\hs{k}_1$ are transverse,
\item if $\hs{h}_0 < \hs{k}_0$ then  $\hs{h}_1 < \hs{k}_1$.
\end{itemize}

For the first bullet point, assume $\hs{h}_0$ and $\hs{k}_0$ are transverse. Then, from the definition of the partial order of a quotient, two cases may occur.
\begin{itemize}
 \item There are preimages in $\Hs{H}$ that are transverse, in which case, by construction of $\Hs{H}_\sim$ and Property \ref{AMcrossinghyps} of $f_1$, the halfspaces $\hs{h}_1$ and $\hs{k}_1$ are transverse. 
\item There are two pairs of halfspaces $(\hs{h}, \hs{k})$ and $(\hs{h}',\hs{k}')$ in $f_0^{-1}(\hs{h}_0)\times f_0^{-1}(\hs{k}_0)$ that are \inseparable{($f_0^{-1}(\hs{h}_0)\cup f_0^{-1}(\hs{k}_0)$)} with contradictory orientations.
But then their images $(\hs{h}_\sim, \hs{k}_\sim)$ and $(\hs{h}'_\sim,\hs{k}'_\sim)$ are transverse or inseparable with contradictory orientations, therefore $\hs{h}_1$ and $\hs{k}_1$ are transverse.
\end{itemize}

For the second bullet point, let $\hs{h}_1$ and $\hs{k}_1$ be the elements associated to $\hs{h}_0$ and $\hs{k}_0$. We want to show that $\hs{h}_1 <\hs{k}_1$.

By definition $\hs{h}_1 <\hs{k}_1$ if no element of $f_1^{-1}(\hs{h}_1)$ is transverse to an element of $f_1^{-1}(\hs{k}_1)$ in  $\Hs{H}_\sim$ and, for any elements $\hs{h}_\sim \in f_1^{-1}(\hs{h}_1)$ and $\hs{k}_\sim \in f_1^{-1}(\hs{k}_1)$ with \inseparable{($f_1^{-1}(\hs{h}_1)\cup f_1^{-1}(\hs{k}_1)$)} hyperplanes, we have $\hs{h}_\sim < \hs{k}_\sim$.

The fact that no elements in $f_1^{-1}(\hs{h}_1)$ and $f_1^{-1}(\hs{k}_1)$ are transverse is a direct application of Lemma \ref{no more crossing}.

Let $\hs{h}_\sim \in f_1^{-1}(\hs{h}_1)$ and $\hs{k}_\sim \in f_1^{-1}(\hs{k}_1)$ be \inseparable{$f_1^{-1}(\hs{h}_1)\cup f_1^{-1}(\hs{k}_1)$}. First notice that no elements of  $\phi^{-1}(\hs{h}_\sim)$ and $\phi^{-1}(\hs{k}_\sim)$  in $\Hs{H}$ are transverse. This is direct since $f_0$ is a resolution and $f_0(\hs{h})<f_0(\hs{k})$.

 It is sufficient to show that $\hyp{h}_\sim \subset \hs{k}_\sim$, then by symmetry of the argument $\hyp{k}_\sim \subset \comp{\hs{h}}_\sim$.
 
Let $\hs{h} \in \phi^{-1}(\hs{h}_\sim)$ and $\hs{k} \in \phi^{-1}(\hs{k}_\sim)$ be \inseparable{($\phi^{-1}(\hs{h}_\sim)\cup \phi^{-1}(\hs{k}_\sim)$)}.
Let $\hs{k}'$ be an element of $f_0^{-1}(\hs{k}_0)$, such that $\hyp{k}'$ and $\hyp{h}$ are \inseparable{$f_0^{-1}(\hyp{k}_0)$}, and is between $\hyp{h}$ and $\hyp{k}$ or equal to $\hyp{k}$. 
Since $\hs{h}_0 <\hs{k}_0$, we have $\hyp{h}\subset \hs{k}'$, and thus by inseparability $\hyp{h}_\sim \subset \hs{k}'_\sim$.

Applying Lemma \ref{thirdguy}, the hyperplanes $\hyp{k}_\sim,\hyp{k}'_\sim, \hyp{h}_\sim$ form a facing triple.
As $\hyp{h}_\sim \subset \hs{k}'_\sim$, we have $\hyp{k}_\sim \subset \hs{k}'_\sim$. 
Since $\hs{k}_\sim$  and $\hs{h}_\sim$ and $\hs{k}'_\sim$  and $\hs{h}_\sim$ are \inseparable{($f_1^{-1}(\hs{h}_1)\cup f_1^{-1}(\hs{k}_1)$)}, the hyperplanes $\hs{k}_\sim$  and $\hs{k}'_\sim$ also are \inseparable{($f_1^{-1}(\hs{h}_1)\cup f_1^{-1}(\hs{k}_1)$)}. And since $f_1$ is admissible, the halfspace $\hs{k}_\sim$ is oriented such that $\hyp{k}'_\sim \subset \hs{k}_\sim$, and thus $\hyp{h}_\sim \subset \hs{k}_\sim$.

Hence we have $\Hs{H}_0 =  \Hs{H}_1$.
\end{proof}

Lemma \ref{one factorization} together with Remark \ref{rmk: an elementary fold exists} show that as long as the resolution $f$ is not an embedding, it admits an elementary fold $\phi$, and it can be decomposed as $f=f_1 \circ \phi$ where $f_1$ is a resolution. 
Iterating this gives us a factorization of $f$ into a sequences of folds.

Proposition \ref{finite_folding} shows, under some conditions, that if we choose the folds correctly then this process terminates.



\begin{proposition}\label{finite_folding}
Let $G$ be a group, let $\Hs{H},\Hs{H}'$ be $G$-pocsets, and let $f:\Hs{H}\to\Hs{H}'$ be a $G$-equivariant resolution. If $G$ has only finitely many orbits in $\Hs{H}$, and the $G$-stabilizers of hyperplanes in $\Hyp{H}'$ are finitely generated, then there exists a finite folding sequence $$ \Hs{H}=\Hs{H}_0 \leadsto \Hs{H}_1 \leadsto \ldots \leadsto \Hs{H}_n=\Hs{H}/\hspace{-.15cm}\sim_f\embedsin\Hs{H}'$$
\end{proposition}

Let us first prove the following lemma.
\begin{lemma} \label{folding two guys}
Let $G$ be a group, let $\Hs{H},\Hs{H}'$ be $G$-pocsets, and let $f:\Hs{H}\to\Hs{H}'$ be a $G$-equivariant resolution. Given two halfspaces $\hs{h}$ and $\hs{k}$ of $\Hs{H}$ that have the same image under $f$, then there exists a finite sequence of folds $ \Hs{H}\leadsto \Hs{H}_1 \leadsto \ldots \leadsto \Hs{H}_n$ such that $\hs{h}$ and $\hs{k}$ are identified in $\Hs{H}_n$.
\end{lemma}

\begin{proof}
Let $\Hyp{I}$ be the collection of hyperplanes between $\hyp{h}$ from $\hyp{k}$ including $\hyp{h}$ and $\hyp{k}$. We prove the proposition by induction on the number of pairs in $\Hyp{I}$ that have the same image under $f$. 
If $\hyp{h}$ and $\hyp{k}$ is the only such pair, then they are elementary foldable, and we get $[\hyp{h}]=[\hyp{k}]$ after one fold.
Otherwise, there exists some pair of hyperplanes in $\Hyp{I}$ that are elementary foldable.
Let $\Hs{H}/\hspace{-.15cm}\sim$ be the folded pocset, and let $f_1: \Hs{H}/\hspace{-.15cm}\sim \to \Hs{H}'$ be the new resolution and $\Hyp{I}_1$ be the collections of hyperplanes between  $[\hyp{h}]$ from $[\hyp{k}]$. 
By Lemma \ref{unfolding2} the image $\Hyp{I}$ has strictly less pairs that have the same image under $f_1$. 
Thus they can be identified after finitely many folds by the induction hypothesis.
\end{proof}

\begin{proof}[Proof of Proposition \ref{finite_folding}]
We first show that by a finite folding sequence one can get $\Hs{H}_i$ such that $\Hs{H}_i/G\to\Hs{H}'/G$ is injective. If this map is not injective, choose two hyperplanes $\hyp{h},\hyp{k}$ such that $f(\hyp{h})=f(\hyp{k})$ but $\hyp{h}$ and $\hyp{k}$ belong to different $G$-orbits. 
By  Lemma \ref{folding two guys}, we can find a sequence of folds that identifies $\hs{h}$ and $\hs{k}$, and thus reduces the number of $G$-orbits of hyperplanes.

Now, we may assume that the map $f/G : \Hs{H}/G\to\Hs{H}'/G$ is injective. For each $G$-orbit of a hyperplane $G.\hyp{h}'$ in its image we fix a hyperplane $\hyp{h}'$ in $\Hs{H}'$ that belongs to this orbit and a finite set $S_{\hyp{h}'}$ of generators of the hyperplane stabilizer $\Stab_G(\hyp{h}')$. We choose some representative $\hyp{h}\in\Hs{H}$ that is mapped by $f$ to $\hyp{h}'$. Let $c_{\hyp{h}'}$ be the number of generators in $S_{\hyp{h}'}$  that do not belong to $Stab_G(\hyp{h})$. Let us define the \emph{complexity} of $f$ to be $$c_f = \sum_{\hyp{h}'\in\Hs{H}'/G} c_{\hyp{h}'}.$$

We prove that if $c_f >0$ then by a finite folding sequence one can reduce $c_f$. Let $\hyp{h}'$ be such that $c_{\hyp{h}'}>0$, and let $S_{\hyp{h}'}$ and $\hyp{h}$ be as in the definition of $c_{\hyp{h}'}$, and let $s\in S_{\hyp{h}'}$ be a generator of $\Stab_G (\hyp{h}')$ that does not belong to $\Stab_G (\hyp{h})$. Note that $f(\hyp{h})=f(s\hyp{h})$ thus by 
Lemma \ref{folding two guys}, we can perform finitely many elementary folds until $\hyp{h}$ and $s\hyp{h}$ are identified. The stabilizer of the resulting hyperplane contains $s$ thus reduces the complexity by at least 1.

We complete the proof by observing that if $f/G$ is injective and $c_f=0$ then $f$ is an embedding of pocsets. 
\end{proof}

Finally, we show that a cobounded action remains cobounded after folding. Note that this is immediate for classical $G$-tree folds, however this becomes less clear in our setting since the dimension can increase. 

\begin{lemma} \label{cobounded}
Let $\gp{G}$ be a group acting coboundedly without inversion on a \CCC $\CC{X}$. Let $\CC{X}'$ be a \CCC obtained by a $\gp{G}$-equivariant elementary fold $\phi$ of $\Hs{H}(\CC{X})$ and assume that $\CC{X}'$ has finite dimension. Then the action of $\gp{G}$ on $\CC{X}'$ is cobounded.
\end{lemma}

\begin{proof}
Let $F:  \CC{X} \rightarrow \CC{X}'$ be the map between \CCC associated to $\phi$. From Lemma \ref{AMtoCCCmap}, the map $F$ is distance non-increasing. 
It is thus sufficient to prove that $\CC{X}'$ is at bounded distance from the image of $F$. 
We will show that any maximal cube in $\CC{X}'$ contains at least a vertex in the image, which would imply that that the $L_1$ distance between $\CC{X}'$ and $F(\CC{X})$ is bounded by the dimension of $\CC{X}'$.

For a maximal cube $\CCc{C}'$ in $\CC{X}'$, we associate a DCC ultrafilter $U$ on $\Hs{H}$ whose corresponding vertex $\CCv{x}\in\CC{X}$ satisfies  $F(\CCv{x})\in\CCc{C}'$.
Let $\Hyp{C}'$ be the set of hyperplanes that cross $\CCc{C}'$.
We partition the set of hyperplanes in $\Hyp{C}'$ into two disjoint subsets: the set $\Hyp{C}'_{NF}$ of non-folded hyperplanes in $\Hyp{C}'$ (i.e, hyperplanes that have one preimage under $\phi$), and the set $\Hyp{C}'_{F}$ of folded hyperplanes in $\Hyp{C}'$. 
Let $\Hyp{C}, \Hyp{C}_{NF}, \Hyp{C}_{F}$ be the corresponding sets of preimages.

We define the ultrafilter $U$ according to the following cases:

\begin{itemize}
	\item For $\hyp h$ in $\Hyp{C}_{NF}$, we set $U(\hyp{h})$ to be an arbitrary choice of orientation of $\hyp h$.
	\item For each $\hyp h$ in $\Hyp{C}_{F}$, by lemma \ref{orientation}, there is an orientation for the hyperplanes in  $\phi^{-1}(\phi(\hyp{h}))$ for which they are facing. We set $U(\hyp{h})$ to be this orientation.
	\item For each $\hyp{h}$ in $\Hyp{H}\setminus\Hyp{C}$, since  $\CCc{C}'$ is maximal, $\phi(\hyp{h})$ is disjoint from some hyperplane $\hyp{k}'\in\Hyp{C}'$. We choose $U(\hyp{h})$ to be the orientation $\hs{h}$ of $\hyp{h}$ that contains the preimages $\phi^{-1}(\hyp{k}')$. Such an orientation exist since otherwise the hyperplane $\hyp{h}$ would be transverse to a preimage of $\hyp{k}'$ or separate two preimages of $\hyp{k}'$. But then, by the converse implication of \Obs \ref{unfolding1}, $\phi(\hyp{h})$ and $\hyp{k}'$ would be transverse, contradicting the hypothesis.
\end{itemize}



Let us first check that $U$ is an ultrafilter of $\CC{X}$. Since any hyperplane in $\Hyp{H}\setminus\Hyp{C}$ is oriented such that it contains all preimages of a hyperplane in $\Hyp{C}'$, it is sufficient to check that the orientations on hyperplanes in $\Hyp{C}$ are compatible. 
Now take $\hyp h'$ and $\hyp{k'}$  in $\Hyp{C}'$. 
By assumption, the hyperplanes $\hyp h'$ and $\hyp k'$ are transverse. 
By \Obs \ref{unfolding1}, either there exists $\hyp k$ in the preimage of $\hyp k'$ that is transverse to a preimage  $\hyp h$ of $\hyp h'$, in which case, since all orientations of $\phi^{-1}(\hyp h')\cup \phi^{-1}(\hyp k')$ contain $\hyp k \cap \hyp h$, any pair of preimages is oriented in a compatible way, or, up to interchanging $\hyp h'$ and $\hyp k'$, there are two hyperplanes $\hyp k_1$ and $\hyp k_2$ in $\phi^{-1}(\hyp k')$ that are separated by a hyperplane $\hyp h\in \phi^{-1}(\hyp h')$. So $U(\hyp k_1)$ and $U(\hyp k_2)$ are facing and contain $\hyp h$. 
By construction any hyperplane in $\phi^{-1}(\hyp h')$ is oriented to contain $\hyp h$, and hyperplanes in $\phi^{-1}(\hyp k')$ are oriented such that they are pairwise facing, in particular they contain $\hyp k_1$ and $\hyp k_2$, and thus also $\hyp h$. This shows that any two hyperplanes in $\phi^{-1}(\hyp h')\cup \phi^{-1}(\hyp k')$ have compatible orientations in $U$.

Moreover $U$ satisfies the DCC condition. Indeed, assume by contradiction that there is an infinite descending chain $\left(\hs{h}_n\right)$. By construction, each $\hs{h}_n$ is oriented towards all preimages of an element of $\Hyp{C}'$. As there are finitely many elements of $\Hyp{C}'$, there exist a preimage of an element of $\Hyp{C}'$ that is contained in every halfspace $\hs{h}_n$ (we have $\hs{h}_n \subset \hs{h}_{n-1}$), contradicting the local finiteness given by Lemma \ref{lem: quotient by AER gives pocset}.

Let $\CCv{x}$ be the vertex corresponding to the ultrafilter $U$. Let us prove now that $F(\CCv{x})$ belongs to $\CCc{C}'$. We have to show that any hyperplane $\hyp h'$ not in $\CCc{C}'$ is oriented toward a hyperplane of $\CCc{C}'$ in the ultrafilter defined by $F(\CCv{x})$. Take the hyperplane $\hyp h$ in $\phi^{-1}(\hyp h')$ such that $\CCv{x}(\hyp h)$ is minimal in $\CCv{x}$. By construction there exists  $\hyp k' \in \Hyp{C}'$ such that $U(\hyp h)$ contains all preimages of $\hyp k'$. Hence, $F(\CCv{x})(\hyp h)$ contains  $\hyp k'$.
\end{proof}

\section{Undistorted subgroups}

In this section we prove the two main results regarding undistorted (and quasiconvex) group actions on CAT(0) cube complexes.
We begin this section by recalling the construction of the geometric resolution, and provide the setting for the remainder of the section.

Let $G$ be a finitely presented group, let $\simp{K}$ be a finite two-dimensional simplicial complex such that $G=\pi_1 (\simp{K})$, and let $\usimp{K}$ be its universal cover. 
Let $G$ act on a finite dimensional CAT(0) cube complex $\CC{X}'$. Without loss of generality we may assume that $G$ acts without inversions.
Since $G\actson \usimp{K}^{(0)}$ freely, there exists a $G$-equivariant map $\usimp{K}^{(0)}\to \CC{X}'^{(0)}$.
Extending this map $G$-equivariantly by mapping edges to combinatorial geodesics and simplices to area minimizing discs, we get a $G$-equivariant map $\usimp{K}\to\CC{X}'$. 
The connected components of preimages of the hyperplanes of $\CC{X}'$ are embedded graphs in $\usimp{K}$, which we call \emph{tracks}.
Each track separates $\usimp{K}$ into two connected components. 
Therefore, the collection of all tracks defines a natural pocset structure (associated to a system of walls) which we denote by $\Hs{H}$.
In this pocset structure each hyperplane $\hyp{h}\in\Hyp{H}$ is associated to a track which we denote by $\trk{t}_{\hyp{h}}$.
Moreover there is a natural resolution of pocsets $\Hs{H}\to\Hs{H}'$, where $\Hs{H}'$ is the pocset of halfspaces of $\CC{X}'$.

The collection of tracks is invariant under the action of $G$, and thus, descends to a collection of immersed graphs in $\simp{K}$. 
Since $\simp{K}$ is finite, this collection is finite. This shows that $G$ has finitely many orbits of hyperplanes in $\Hyp{H}$. Thus, if we further assume that for all $\hyp{h}'\in\Hyp{H}'$ the stabilizer $\Stab_G(\hyp{h}')$ is finitely generated then by Proposition \ref{finite_folding} there is a finite folding sequence
\[ \Hs{H}=\Hs{H}_0 \leadsto \Hs{H}_1 \leadsto \ldots \leadsto \Hs{H}_n=\Hs{H}/\hspace{-.15cm}\sim_f\embedsin\Hs{H}'\]
factoring $\Hs{H}\to\Hs{H}'$.
Let us denote by 
\[ \CC{X}=\CC{X}_0 \leadsto \CC{X}_1 \leadsto \ldots \leadsto \CC{X}_n=\CC{X}/\hspace{-.15cm}\sim_f\embedsin\CC{X}'\] 
the corresponding sequence of CAT(0) cube complexes.

Moreover, since $\simp{K}$ is finite each track has a finite (immersed) image in $\simp{K}$. Thus, it is easy to see that $\trk{t}_{\hyp{h}}$ is $\Stab_G(\hyp{h})$-invariant (in fact, $\Stab_G(\trk{t}_{\hyp{h}})=\Stab_G(\hyp{h})$) and that $\Stab_G(\hyp{h})$ acts cocompactly on $\trk{t}_{\hyp{h}}$. 
However, it might not act cocompactly on $\hyp{h}$.
In the following claim we generalize these properties of tracks to hyperplanes in $\Hyp{H}_i$.
To do that we extend tracks to immersed graphs (which are not necessarily tracks) which we call ``saturated tracks''.

\begin{claim}\label{saturated tracks}
For each $1\le i\le n$ there is a graph $T_i$ with an action of $G$ and an immersion $T_i\immersedin \usimp{K}$ such that the following holds. 
\begin{enumerate}
	\item \label{strk contain tracks} The graph  $T_0$ is the disjoint union of the tracks of the geometric resolution, that is, $T_0=\coprod_{\hyp{h}\in\Hyp{H}}\trk{t}_{\hyp{h}}\embeddedin \usimp{K}$ where $\trk{t}_{\hyp{h}}\embeddedin\usimp{K}$ is the track associated to $\hyp{h}\in\Hyp{H}=\Hyp{H}_0$.
	\item \label{strk are increasing} For all $1\le i\le n$, $T_{i-1}\subseteq T_{i}$ and both the $G$-action and the immersion of $T_i$ extend that of $T_{i-1}$.
	\item \label{strk are G cocompact and equivariant} $G$ acts cocompactly on $T_i$ and the immersion $T_i\immersedin \usimp{K}$ is $G$-equivariant.
	\item \label{strk are connected components} There is a $G$-equivariant bijection $\Hs{H}_i\ni\hyp{h}^i \mapsto \trk{t}_{\hyp{h}^i}\in\pi_0(T_i)$ from $\Hyp{H}_i$ to $\pi_0(T_i)=$ the connected components of $T_i$.  We call $\trk{t}_{\hyp{h}^i}$ \emph{the $i$-saturated track} of $\hyp{h}^i$. In particular, we have $\Stab_G (\trk{t}_{\hyp{h}^{i}}) = \Stab_G(\hyp{h}^{i})$
	\item \label{strk are stab cocompact} For all $\hyp{h}^i\in\Hyp{H}_i$ the stabilizer $\Stab_G(\hyp{h}^{i})$ acts cocompactly on $\trk{t}_{\hyp{h}^{i}}$.
\end{enumerate} 
\end{claim}
%

\begin{proof}
	We build $T_i$ by induction on $i$. 
	As we have said, for $i=0$ the collection $T_0=\coprod_{\hyp{h}\in\Hyp{H}}\trk{t}_{\hyp{h}}$ satisfies the properties.
	
	Assume $T_i$ has been defined, and let $\hs{h}^{i}_1\sim\hs{h}^{i}_2$ be the elementary foldable pair of halfspaces in $\Hs{H}_i$ that generates the fold $\Hs{H}_i\leadsto \Hs{H}_{i+1}$.
	Let $\trk{t}_{\hyp{h}^{i}_1},\trk{t}_{\hyp{h}^{i}_2}\subseteq T_i$ be the associated $i$-saturated tracks in $T_i$
	and let $\alpha$ be an immersed arc in $\usimp{K}$ connecting a point of $\trk{t}_{\hyp{h}^{i}_1}$ and a point of $\trk{t}_{\hyp{h}^{i}_2}$ in $\usimp{K}$. 
	Let $T_{i+1}$ be obtained from $T_i$ by $G$-equivariantly adding an edge $\tilde{\alpha}$ between the endpoints of $\alpha$ in $\trk{t}_{\hyp{h}^{i}_1}$ and $\trk{t}_{\hyp{h}^{i}_2}$. Let $T_{i+1}\immersedin \usimp{K}$ extend the map $T_i\immersedin \usimp{K}$ $G$-equivariantly such that the edge $\tilde{\alpha}$ is mapped to $\alpha$. By construction $T_i\subseteq T_{i+1}$ comes with a cocompact $G$-action, and an immersion $T_{i+1}\immersedin\simp{K}$ which is $G$-equivariant and extends $T_i\immersedin \usimp{K}$, proving properties \ref{strk are increasing} and \ref{strk are G cocompact and equivariant}.
		
	The equivalence relation on hyperplanes that produces the quotient $\Hyp{H}_i\leadsto \Hyp{H}_{i+1}$ is generated by $G$-equivariantly identifying the two hyperplanes $\hyp{h}^{i}_1\sim\hyp{h}^{i}_2$. Similarly, the equivalence relation on $i$-saturated tracks (i.e, components of $T_i$) in which two $i$-saturated tracks are equivalent if they belong to the same component of $T_{i+1}$ is generated by $G$-equivariantly identifying $\trk{t}_{\hyp{h}^{i}_1}$ and $\trk{t}_{\hyp{h}^{i}_2}$ (as this is exactly the effect of connecting them with the edge $\tilde{\alpha}$).
	Thus the following is a well defined bijection. 
	The $(i+1)$-saturated track $\trk{t}_{\hyp{h}^{i+1}}$ associated to a hyperplane $\hyp{h}^{i+1}\in\Hyp{H}_{i+1}$ is the connected component of $T_{i+1}$ that contains the $i$-saturated track $\trk{t}_{\hyp{h}^i}$ of a preimage $\hyp{h}^{i}$ of $\hyp{h}^{i+1}$ under the map $\Hyp{H}_i \leadsto \Hyp{H}_{i+1}$.
	Clearly, this bijection is $G$-equivariant which proves property \ref{strk are connected components}.
	
	Finally, notice that since the saturated tracks are connected components, if $g\in G$ is such that $g\trk{t}_{\hyp{h}^i} \cap \trk{t}_{\hyp{h}^i} \ne \emptyset$ then $g\in \Stab_G(\trk{t}_{\hyp{h}^i}) = \Stab_G (\hyp{h}^i)$. Therefore, property \ref{strk are stab cocompact} follows from property \ref{strk are G cocompact and equivariant}.
\end{proof}

\begin{claim}\label{intersecting strk}
	If $\hyp{h}^{i}\pitchfork\hyp{k}^{i}\in\Hyp{H}_i$ then 
	the images of the immersions of $\trk{t}_{\hyp{h}^{i}}$ and $\trk{t}_{\hyp{k}^{i}}$ intersect in $\usimp{K}$.
\end{claim}

\begin{proof}
%
	If $\hyp{h}^{i}\pitchfork\hyp{k}^{i}\in\Hyp{H}_i$ then from Lemma \ref{preimage of crossing hyps} there are two possible cases for their preimages under the map $\Hyp{H}\to\Hyp{H}_i$. There are intersecting preimages $\hyp{h},\hyp{k}\in \Hyp{H}$ of $\hyp{h}^{i},\hyp{k}^i$ respectively, which implies that $\trk{t}_{\hyp{h}}$ and $\trk{t}_{\hyp{k}}$ intersect, and thus also $\trk{t}_{\hyp{h}^i}$ and $\trk{t}_{\hyp{k}^i}$. 
	Or, up to exchanging $\hyp{h}^i,\hyp{k}^i$ there is a preimage $\hyp{h}$ of $\hyp{h}^i$ that separates two preimages $\hyp{k}_1,\hyp{k}_2$ of $\hyp{k}^i$. 
	In this case, the corresponding track $\trk{t}_{\hyp{h}}$ separates the tracks $\trk{t}_{\hyp{k}_1},\trk{t}_{\hyp{k}_2}$. 
	But since $\trk{t}_{\hyp{h}}\subseteq \trk{t}_{\hyp{h}^i}$ and $\trk{t}_{\hyp{k}_1},\trk{t}_{\hyp{k}_2}\subseteq \trk{t}_{\hyp{k}^i}$, and $\trk{t}_{\hyp{k}^i}$ is connected we see that $\trk{t}_{\hyp{h}^i}$ and $\trk{t}_{\hyp{k}^i}$ intersect.
\end{proof}

\begin{theorem} \label{maintheorem1}
Let $G$ be a finitely presented group that acts properly on a CAT(0) cube complex $\CC{X}'$, with finitely generated hyperplane stabilizers
and such that the action on the geometric resolution is cobounded. Then the orbit maps $G\to \CC{X}'$ are quasi-isometric embeddings.
\end{theorem}

\begin{proof}
	Following the above discussion, we see that there is a folding sequence  \[ \CC{X}=\CC{X}_0 \leadsto \CC{X}_1 \leadsto \ldots \leadsto \CC{X}_n=\CC{X}/\hspace{-.15cm}\sim_f\embedsin\CC{X}'.\]
	By applying inductively Lemma \ref{cobounded}, $G$ acts coboundedly on each of $\CC{X}_i$, and in particular on $\CC{X}_n$. 
	Since $G$ acts properly on $\CC{X}'$ and the map $\CC{X}_n\embedsin \CC{X}'$ is combinatorial, $G$ acts properly on $\CC{X}_n$. 
	The action of $G$ on $\CC{X}_n$ is proper and cobounded, and therefore $\CC{X}_n$ is $G$-equivariantly quasi-isometric to $G$. The $G$-equivariant embedding $\CC{X}_n\embedsin \CC{X}$ completes the proof.
\end{proof}

%
As a Corollary we obtain Theorem \ref{thm: undistorted in 2D}.

\begin{proof}[Proof of Theorem \ref{thm: undistorted in 2D}] Using  theorem \ref{maintheorem1}, we only need to show that in each of these cases, the action on the resolution is cobounded.
	Case \ref{surface_group_case} follows from \cite{Sag95}.
	Case \ref{2_dim_case} follows from Theorem 1.1 of \cite{BeLa15}. 
\end{proof}

The second theorem gives equivalences  between different quasiconvex properties of subgroups and the quasiconvexity of the group. For this we need the notion of \emph{cocompact core}.

\begin{definition}[\condX]
\label{condition X}
A finitely presented group $H$ acting properly on a CAT(0) cube complex $\CC{X'}$ satisfies \condX if $H$ acts cocompactly on a \CCC $\CC{Y}$ such that $\CC{Y} \embedsin \CC{X'}$ $H$-equivariantly. 
\end{definition}

For the remainder of this section we assume that $G$ is Gromov hyperbolic. The following lemma can be viewed as a generalization of the thin triangle condition for higher dimensional simplices.

\begin{definition}
For $R\ge 0$, a collection of subsets $A_1,\ldots,A_n$ is \emph{$R$-coarsely intersecting} if their $R$-neighborhoods have a non-empty common intersection. 
A collection of subsets $A_1,\ldots,A_n$ is \emph{pairwise $R$-coarsely intersecting} if any pair is $R$-coarsely intersecting.
\end{definition}

The following lemma appears as Lemma 7 in \cite{NiRe03} in the case of convex subsets. Even though their proof works in the quasiconvex case, we include a proof for the purpose self-containment.

\begin{lemma} [The Thin Simplex Lemma]
\label{lem: thin simplex}
Let $H$ be a $\delta$-hyperbolic geodesic space, let $A_1,\ldots, A_d$ be $R$-quasiconvex subsets for some $R$, and assume that $A_i$ pairwise $R$-coarsely intersect. Then there exists $R'=R'(\delta, d,R)$ such that $A_1,\ldots, A_d$ $R'$-coarsely intersect.
\end{lemma}

\begin{proof}

Let $x_{i,j}$ denote an intersection point of the $R$-neighborhoods of $A_i$ and $A_j$. 
By \cite[Proposition 3.2]{Bow98}, there exists a finite metric tree $T$ with distinguished points $y_{i,j}$ ($1\le i,j\le d$), and a quasi-isometric embedding $\psi:T\to H$ such that $\psi (y_{i,j})= x_{i,j}$ and the quasi-isometry constants depend only on $\delta$ and $d$.
For each $1\le i \le d$, let $T_i$ be the subtree of $T$ spanned by $\{ y_{i,j} | 1\le j \le d\}$. Since the subtrees $T_i$ pairwise intersect, by the Helly property there exists an intersection point $y\in \bigcap _{i=1} ^d T_i$.
The image $y$ under $\psi$ is on the quasiconvex sets $\psi (T_i)$, which by the quasiconvexity of $A_i$ and the stability of quasi-geodesic, are at bounded distance $R'=R'(\delta,d,R)$ from each $A_i$.
\end{proof}

When applying the previous lemma on cosets of quasiconvex subgroups we can deduce the following.

\begin{lemma}
\label{lem: coarsely intersecting implies cond X}
If $H$ is a hyperbolic group and $L_k,\;k=1,\ldots,r$, are $R$-quasiconvex subgroups. Then for any $d$, $H$ acts co-finitely on collections of $d$ pairwise $R$-coarsely intersecting cosets of $L_k$.
\end{lemma}

\begin{proof}
From the Thin Simplex Lemma there exists $R'$ such that any collection of pairwise $R$-coarsely intersecting cosets of $L_k,\; 1\le k\le r$ must intersect some $R'$ ball. 
Since $H$ acts on itself coboundedly, up to the action of $H$ there are only finitely many such balls, and thus only finitely many collections of $d$ cosets that intersect them.
\end{proof}

We are now ready to prove Theorem \ref{thm:QC}. In fact, we prove the following version.

\begin{theorem}
\label{thm:QC2}
Let $G$ be a hyperbolic group acting properly and co-compactly on a finite dimensional CAT(0) cube complex $\CC{X}'$. Let $H\le G$ be a finitely presented subgroup. Then the following are equivalent:
\begin{enumerate}
\item \label{QC} The subgroup $H$ is quasiconvex in $G$.
\item \label{QC by interated hyperplanes} For all $\hyp{t}\in \Hyp{IH}'$ (see definition in the introduction), the group $\Stab_H (\hyp{t})$ is finitely presented.
\item \label{QC by qc in G}For all $\hyp{k}\in\Hyp{H}'$, $\Stab_H (\hyp{k})$ is quasiconvex in $G$.
\item \label{QC by qc in H}The subgroup $H$ is hyperbolic and for all $\hyp{k}\in\Hyp{H}'$, $\Stab_H (\hyp{k})$ is quasiconvex in $H$.
\item \label{QC by condition X}The subgroup action $H\actson \CC{X}'$ satisfies \condX.

\end{enumerate}
\end{theorem}

\begin{proof}
We recall the following 4 facts:
\begin{enumerate}[label=(\emph{\alph*})]
\item \label{inter qc is qc} the intersection of (finitely many) quasiconvex subgroups is quasiconvex,
\item \label{qc in hyp is hyp} a quasiconvex subgroup of a hyperbolic group is itself hyperbolic, and therefore finitely presented,
\item \label{qc in sbgp iff in gp} a subgroup of a quasiconvex subgroup is quasiconvex in the subgroup if and only if it is quasiconvex in the ambient group, and
\item \label{hyps are qc} if a hyperbolic group $G$ acts properly co-compactly on a CAT(0) cube complex then its hyperplane stabilizers $\Stab_G (\hyp{h})$ are quasiconvex in $G$.
\end{enumerate}

From the above facts it is easy to see that \ref{QC} implies \ref{QC by interated hyperplanes}, \ref{QC by qc in G} and \ref{QC by qc in H}.
Indeed from \ref{hyps are qc}, hyperplanes stabilizers in $G$ are quasiconvex in $G$. From \ref{inter qc is qc}, elements in $\Hyp{IH}'$ are quasiconvex, and from \ref{qc in hyp is hyp}, they are finitely presented. Therefore \ref{QC} implies \ref{QC by interated hyperplanes} and \ref{QC by qc in G}. Using \ref{qc in sbgp iff in gp}, we deduce  \ref{QC} $\implies$ \ref{QC by qc in H}.

The implication \ref{QC by condition X}$\implies$\ref{QC} is also immediate.

For the remaining implication we proceed as follows. We first prove that \ref{QC by qc in G}$\iff$\ref{QC by qc in H}$\iff$\ref{QC by condition X}$\iff$\ref{QC}, by showing \ref{QC by qc in H}$\implies$\ref{QC by condition X} and \ref{QC by qc in G}$\implies$\ref{QC by condition X}. We then use the implication \ref{QC by qc in G}$\implies$\ref{QC} to show \ref{QC by interated hyperplanes}$\implies$\ref{QC}.

	In the remaining cases we are in the setting of the discussion at the beginning of the section, where $G$ is replaced by the subgroup $H$ (i.e, $H=\pi_1(\simp{K})$ etc.).
	Therefore, we have a geometric resolution $\CC{X}'\to \CC{X}$ for the action of $H$ on $\CC{X}$ and a folding sequence \[ \CC{X}=\CC{X}_0 \leadsto \CC{X}_1 \leadsto \ldots \leadsto \CC{X}_n=\CC{X}/\hspace{-.15cm}\sim_f\embedsin\CC{X}'.\]
	Let $\hyp{h}^n_1,\ldots,\hyp{h}^n_r$ be a set of $H$-orbit representatives for $H\actson\Hyp{H}_n$ (recall that the action of $H$ on $\Hyp{H}_n$ has finitely many orbits).
	Let us denote by $L_k = \Stab_H (\hyp{h}^n_k)$ for all $1\le k \le r$.	
	By Property \ref{strk are stab cocompact} of Claim \ref{saturated tracks} for all $1\le k \le r$, $L_k$ acts cocompactly on the $n$-saturated track $\trk{t}_{\hyp{h}^n_k}$. If we fix $x_0\in\usimp{K}$, then there exists $R'\ge 0$ such that the orbit $L_k.x_0$ is at Hausdorff distance at most $R'$ from the image of $\trk{t}_{\hyp{h}^n_k}$ in $\usimp{K}$ for all $1\le k \le r$.  
	It follows that the set $gL_k.x_0$ is at Hausdorff distance at most $R'$ from $g.\trk{t}_{\hyp{h}^n_k}$ in $\usimp{K}$ for all $g\in G$ and $1\le k \le r$.
	Therefore, by Claim \ref{intersecting strk} if $g\hyp{h}_k^n\pitchfork g'\hyp{h}_{k'}^n$ then the corresponding $n$-saturated tracks $g\trk{t}_{\hyp{h}_k^n}$ and $g'\trk{t}_{\hyp{h}_{k'}^n}$ intersect, and the corresponding sets $gL_k.x_0$ and $g'L_{k'}.x_0$ are $2R'$-coarsely intersecting.
	Notice that the orbit map $H\to\usimp{K}$ defined by $g\mapsto g.x_0$ is a quasi-isometric, in particular there exists $R\ge 0$ such that if $gL_k.x_0$ and $g'L_{k'}.x_0$ are $2R'$-coarsely intersecting, then $gL_k$ and $g'L_{k'}$ are $R$-coarsely intersecting (for all $g,g'\in G$ and $1\le,k,k'\le r$). 
	
	To summarize this discussion, there exists $R$ such that a collection of transverse hyperplanes in $\Hyp{H}_n$ corresponds to a pairwise $R$-coarsely intersecting collections of cosets of the hyperplane stabilizers $L_k$ in $H$.
	
	\ref{QC by qc in H}$\implies$\ref{QC by condition X}: Since $H$ is hyperbolic and $L_i$ are quasiconvex, we can apply Lemma \ref{lem: coarsely intersecting implies cond X} to show that $H$ acts cofinitely on cubes of $\CC{X}_n$. This implies that it acts cocompactly on $\CC{X}_n$. 
	Finally, since $\CC{X}_n\embedsin\CC{X}'$ and $H$ acts properly and cocompactly on $\CC{X}_n$ we have shown that $H$ has \condX with $\CC{Y} = \CC{X}_n$.
	
%
%
	\ref{QC by qc in G}$\implies$\ref{QC by condition X}: As before, there exists $R$ such that a collection of transverse hyperplanes in $\Hs{H}_n$ corresponds to a pairwise $R$-coarsely intersecting collection of cosets of the hyperplane stabilizers $L_k$ in $H$. 
	Since $H$ is a (finitely generated) subgroup of $G$, there exists $R_2$ such that a pairwise $R$-coarsely intersecting collection of cosets in $H$ is pairwise $R_2$-coarsely intersecting in $G$. 
	Since the hyperplane stabilizers are assumed to be quasiconvex in $G$, and $G$ is hyperbolic, we can apply Lemma \ref{lem: thin simplex} to deduce that there exists $R_3$ such that any such collection is $R_3$-coarsely intersecting in $G$.
	But since $H$, being a finitely generated subgroup of $G$, is coarsely embedded in $G$, there exists $R_4$ such that a collection of subsets of $H$ that is $R_3$-coarsely intersecting in $G$ is $R_4$-coarsely intersecting in $H$, and thus, as in the proof of Lemma \ref{lem: coarsely intersecting implies cond X}, $H$ acts cofinitely on such collections, which again proves \condX with $\CC{Y} = \CC{X}_n$.

Finally we prove \ref{QC by interated hyperplanes}$\implies$\ref{QC} by induction on the dimension of $\CC{X}'$. If $\dim(\CC{X}')=1$, then it follows from the well known fact that a finitely generated subgroup of a group acting properly on a tree is quasiconvex. 
Now let $\dim(\CC{X}')=d+1\ge2$. The hyperplanes of $\CC{X}'$ are CAT(0) cube complexes of dimension  at most $d$. For all $\hyp{k}\in\Hyp{H}(\CC{X}')$ the subgroup $\Stab_H (\hyp{k})$ is  a finitely presented group acting on a $d$-dimensional \CCC with finitely presented intersections of hyperplane stabilizers. Hence, by the induction hypothesis, they are quasiconvex in $\Stab_G (\hyp{k})$, and thus also quasiconvex in $G$. The desired conclusion follows from \ref{QC by qc in G}$\implies$\ref{QC}, that we have already proved.
\end{proof}

\bibliographystyle{plain}
\bibliography{main}

\begin{thebibliography}{10}

\bibitem{Ago13}
Ian Agol.
\newblock The virtual {H}aken conjecture.
\newblock {\em Doc. Math.}, 18:1045--1087, 2013.
\newblock With an appendix by Agol, Daniel Groves, and Jason Manning.

\bibitem{BBD12}
Josh Barnard, Noel Brady, and Pallavi Dani.
\newblock Super-exponential 2-dimensional {D}ehn functions.
\newblock {\em Groups Geom. Dyn.}, 6(1):1--51, 2012.

\bibitem{BeLa15}
Benjamin Beeker and Nir Lazarovich.
\newblock Resolutions of {CAT}(0) cube complexes and accessibility properties.
\newblock {\em Algebr. Geom. Topol.}, 16(4):2045--2065, 2016.

\bibitem{BeWi12}
Nicolas Bergeron and Daniel~T. Wise.
\newblock A boundary criterion for cubulation.
\newblock {\em Amer. J. Math.}, 134(3):843--859, 2012.

\bibitem{BeFe91}
Mladen Bestvina and Mark Feighn.
\newblock Bounding the complexity of simplicial group actions on trees.
\newblock {\em Invent. Math.}, 103(1):449--469, 1991.

\bibitem{Bow98}
Brian~H. Bowditch.
\newblock A topological characterisation of hyperbolic groups.
\newblock {\em J. Amer. Math. Soc.}, 11(3):643--667, 1998.

\bibitem{Bro16}
Samuel Brown.
\newblock {\em Geometric structures on negatively curved groups and their
  subgroups}.
\newblock PhD thesis, University College London, August 2016.

\bibitem{DiRi13}
Will Dison and Timothy~R. Riley.
\newblock Hydra groups.
\newblock {\em Comment. Math. Helv.}, 88(3):507--540, 2013.

\bibitem{Dun85}
Martin~J. Dunwoody.
\newblock The accessibility of finitely presented groups.
\newblock {\em Invent. Math.}, 81(3):449--457, 1985.

\bibitem{Gro87}
Mikhael Gromov.
\newblock {\em Hyperbolic groups}.
\newblock Springer, 1987.

\bibitem{HaWi15}
Mark~F. Hagen and Daniel~T. Wise.
\newblock Cubulating hyperbolic free-by-cyclic groups: the general case.
\newblock {\em Geom. Funct. Anal.}, 25(1):134--179, 2015.

\bibitem{Hag08}
Fr{\'e}d{\'e}ric Haglund.
\newblock Finite index subgroups of graph products.
\newblock {\em Geom. Dedicata}, 135:167--209, 2008.

\bibitem{HaWi08}
Fr{\'e}d{\'e}ric Haglund and Daniel~T. Wise.
\newblock Special cube complexes.
\newblock {\em Geom. Funct. Anal.}, 17(5):1551--1620, 2008.

\bibitem{HsWi10}
Tim Hsu and Daniel~T. Wise.
\newblock Cubulating graphs of free groups with cyclic edge groups.
\newblock {\em Amer. J. Math.}, 132(5):1153--1188, 2010.

\bibitem{HsWi15}
Tim Hsu and Daniel~T. Wise.
\newblock Cubulating malnormal amalgams.
\newblock {\em Invent. Math.}, 199(2):293--331, 2015.

\bibitem{KaMa12}
Jeremy Kahn and Vladimir Markovic.
\newblock Immersing almost geodesic surfaces in a closed hyperbolic three
  manifold.
\newblock {\em Ann. of Math. (2)}, 175(3):1127--1190, 2012.

\bibitem{LaWi13}
Joseph Lauer and Daniel~T. Wise.
\newblock Cubulating one-relator groups with torsion.
\newblock {\em Math. Proc. Cambridge Philos. Soc.}, 155(3):411--429, 2013.

\bibitem{Mit98}
Mahan Mitra.
\newblock Cannon-{T}hurston maps for trees of hyperbolic metric spaces.
\newblock {\em J. Differential Geom.}, 48(1):135--164, 1998.

\bibitem{NiRe98}
Graham~A. Niblo and Lawrence~D. Reeves.
\newblock The geometry of cube complexes and the complexity of their
  fundamental groups.
\newblock {\em Topology}, 37(3):621--633, 1998.

\bibitem{NiRe03}
Graham~A. Niblo and Lawrence~D. Reeves.
\newblock Coxeter groups act on {${\rm CAT}(0)$} cube complexes.
\newblock {\em J. Group Theory}, 6(3):399--413, 2003.

\bibitem{OlWi11}
Yann Ollivier and Daniel~T. Wise.
\newblock Cubulating random groups at density less than {$1/6$}.
\newblock {\em Trans. Amer. Math. Soc.}, 363(9):4701--4733, 2011.

\bibitem{Rip82}
Eliyahu Rips.
\newblock Subgroups of small cancellation groups.
\newblock {\em Bull. London Math. Soc.}, 14(1):45--47, 1982.

\bibitem{Rol98}
Martin~A. Roller.
\newblock Poc sets, median algebras and group actions. an extended study of
  {D}unwoody’s construction and {S}ageev’s theorem.
\newblock Univ. of Southampton, 1998.

\bibitem{Sag95}
Michah Sageev.
\newblock Ends of group pairs and non-positively curved cube complexes.
\newblock {\em Proc. Lond. Math. Soc.}, 3(3):585--617, 1995.

\bibitem{Saa12}
Michah Sageev.
\newblock {CAT}(0) cube complexes and groups.
\newblock Park City Mathematics Institute, 2012.

\bibitem{SaWi15}
Michah Sageev and Daniel~T. Wise.
\newblock Cores for quasiconvex actions.
\newblock {\em Proc. Amer. Math. Soc.}, 143(7):2731--2741, 2015.

\bibitem{Sta83}
John~R. Stallings.
\newblock Topology of finite graphs.
\newblock {\em Invent. Math.}, 71(3):551--565, 1983.

\bibitem{Sta91}
John~R. Stallings.
\newblock Foldings of {$G$}-trees.
\newblock In {\em Arboreal group theory ({B}erkeley, {CA}, 1988)}, volume~19 of
  {\em Math. Sci. Res. Inst. Publ.}, pages 355--368. Springer, New York, 1991.

\bibitem{Wis98}
Daniel~T. Wise.
\newblock Incoherent negatively curved groups.
\newblock {\em Proc. Amer. Math. Soc.}, 126(4):957--964, 1998.

\bibitem{Wis04}
Daniel~T. Wise.
\newblock Cubulating small cancellation groups.
\newblock {\em Geom. Funct. Anal.}, 14(1):150--214, 2004.

\bibitem{Wis11}
Daniel~T. Wise.
\newblock The structure of groups with a quasiconvex hierarchy.
\newblock preprint, 2011.

\end{thebibliography}

\end{document}